\documentclass[a4paper]{amsart}
\usepackage[ansinew]{inputenc}
\usepackage[T1]{fontenc}
\usepackage{longtable}
\usepackage{xcolor}
\usepackage{mathrsfs}
\usepackage{amssymb,amsmath,amsthm}
\usepackage[matrix, arrow, curve]{xy}
\usepackage{tikz-cd}
\usepackage{stmaryrd}
\usepackage{mathtools}
\usepackage{extarrows}
\usepackage{hhline}
\usepackage{hyperref}
\usepackage[nameinlink]{cleveref}
\usetikzlibrary{quotes,babel, angles}

\newcommand{\fin}{\textnormal{fin}\,}
\newcommand{\Proj}{\textnormal{Proj}\,}
\newcommand{\proj}{\textnormal{proj}\,}
\newcommand{\lend}{l_{\textnormal{end}}}

\newcommand{\Ex}{\textnormal{Ex}}

\newcommand{\op}{\textnormal{op}}
\newcommand{\Hom}{\textnormal{Hom}}
\newcommand{\Ext}{\textnormal{Ext}}
\renewcommand{\mod}{\textnormal{mod}\,}

\newcommand{\Mod}{\textnormal{Mod}\,}
\newcommand{\vMod}{\textnormal{Mod}}
\newcommand{\im}{\textnormal{Im}\,}
\newcommand{\coker}{\textnormal{coker}\,}

\newcommand{\gen}{\textnormal{gen}\,}
\newcommand{\cogen}{\textnormal{cogen}\,}
\newcommand{\Ab}{\textnormal{Ab}}

\newcommand{\Ind}{\textnormal{Zg}\,}

\newcommand{\inj}{\textnormal{inj}\,}

\newcommand{\Spec}{\textnormal{Spec}}
\newcommand{\Frac}{\textnormal{Frac}}

\newcommand{\Stab}{\mathrm{Stab}}
\newcommand{\End}{\mathrm{End}}
\newcommand{\cO}{\mathcal{O}}

\numberwithin{equation}{section} \theoremstyle{plain}
\newtheorem*{thm*}{Theorem}
\newtheorem{thm}{Theorem}
\numberwithin{thm}{section}

\newtheorem{coro}[thm]{Corollary}
\newtheorem*{coro*}{Corollary}

\newtheorem*{ass*}{Assumption}
\newtheorem{lem}[thm]{Lemma}
\newtheorem*{lem*}{Lemma}
\newtheorem{prop}[thm]{Proposition}
\newtheorem*{prop*}{Proposition}
\newtheorem{rem}[thm]{Remark}
\newtheorem*{rem*}{Remark}
\newtheorem{exa}[thm]{Example}
\newtheorem*{exa*}{Example}

\newtheorem*{df*}{Definition}

\newtheorem*{ques*}{Question}

\newtheorem*{conj*}{Conjecture}

\newtheorem*{construction*}{Construction}
\newtheorem*{ack*}{ACKNOWLEDGEMENTS}

\begin{document}

\title{Constructible Subcategories and Unbounded Representation Type}
\author{Kevin Schlegel \\With an appendix by Andres Fernandez Herrero}
\date{}

\address{Kevin Schlegel, University of Stuttgart, Institute of Algebra and Number Theory, Pfaffenwaldring 57, 70569 Stuttgart, Germany}
\email{kevin.schlegel@iaz.uni-stuttgart.de}
\address{Andres Fernandez Herrero, University of Pennsylvania, Department of Mathematics,
209 South 33rd Street, Philadelphia, PA 19104, USA}
\email{andresfh@sas.upenn.edu}

\subjclass[2020]{16E30, 16G60, 03C60, 14L30}
\keywords{Brauer-Thrall conjectures, representation type, Ziegler spectrum, scheme of modules, exact structures, matrix reductions}
\begin{abstract} We show that bounded type implies finite type for a constructible subcategory of the module category of a finitely generated algebra over a field, which is a variant of the first Brauer-Thrall conjecture. A full subcategory is constructible if it consists of all modules that vanish on a finitely presented functor. Our approach makes use of the Ziegler spectrum of a ring and a connection, established in this work, with the scheme of finite dimensional modules. We also discuss a variant of the second Brauer-Thrall conjecture in this setting. It is shown to be true under additional assumptions on the algebra, but wrong in general. Lastly, it is proven that exact structures and matrix reductions give rise to constructible subcategories. Based on this, a translation from matrix reductions to reductions of exact structures is provided.    
\end{abstract}
\maketitle 

\section*{Introduction}
The Brauer-Thrall conjectures had a big influence on the development of the representation theory of finite dimensional algebras \cite{Jans}. The first Brauer-Thrall conjecture, first proven by Roiter \cite{Roiter}, states that for a finite dimensional algebra $A$ bounded representation type implies finite representation type. That is, if the dimension of all finite dimensional indecomposable $A$-modules is bounded, then there are only finitely many of them, up to isomorphism. It was shown in \cite{BG} that the statement can be extended to finite dimensional modules over a finitely generated algebra $A$ by reducing it to the finite dimensional case. The main result in this work is to prove a variant of the first Brauer-Thrall conjecture for a wide range of subcategories of the module category $\Mod A$.\\
\\
\textbf{Theorem A} (\Cref{mthm}) \textit{Let $k$ be a field, $A$ a finitely generated $k$-algebra and $\mathcal{X} \subseteq \Mod A$ a constructible subcategory. If $\mathcal{X}$ is of infinite type, then $\mathcal{X}$ is of unbounded type.}\\

A full subcategory $\mathcal{X} \subseteq \Mod A$ is \emph{constructible} if $\mathcal{X} = \{X\in \Mod A \mid F(X) = 0\}$ for a finitely presented functor $F\colon \Mod A \to \Ab$. We say that $\mathcal{X}$ is of \emph{infinite type} if there are infinitely many non-isomorphic finite dimensional indecomposable modules in $\mathcal{X}$, \emph{unbounded type} if there is no bound on the dimension of the finite dimensional indecomposable modules in $\mathcal{X}$, and \emph{strongly unbounded type} if there are natural numbers $n_1 < n_2 < \dots$ such that for all $i>0$ there are infinitely many non-isomorphic indecomposable modules of dimension $n_i$ in $\mathcal{X}$. These are properties of the full subcategory $\fin \mathcal{X} \subseteq \mathcal{X}$ of finite dimensional modules. If $A$ is finite dimensional, then the following are examples for $\fin \mathcal{X}$ with \mbox{$\mathcal{X} \subseteq \Mod A$} constructible: the category $\mod A$ of all finite dimensional $A$-modules, Hom- and Ext$^n$-orthogonals of a finite dimensional module within $\mod A$, subcategories \mbox{(co-)generated} by a finite dimensional module, the full subcategory of $\mod A$ consisting of all modules with projective (injective) dimension bounded by some $n\in \mathbb{N}$, functorially finite torsion(-free) classes in $\mod A$, and $\Delta$-good module categories for quasi-hereditary algebras. This shows the abundance of constructible subcategories.

The main obstacle in the proof of Theorem A is that a constructible subcategory $\mathcal{X} \subseteq \Mod A$ loses most of the nice homological properties of the module category. Thus, the methods of Roiter \cite{Roiter} and Auslander's approach \cite{Auslander0} do not apply. The approach in this work is of more geometric nature. We make use of the Ziegler spectrum of a ring \cite{Ziegler} and show a connection with the scheme of finite dimensional $A$-modules \cite{Morrison}, essentially reducing the problem to the case of the coordinate ring of a regular curve over $k$. The Ziegler spectrum is a topological space that appears in the context of purity, which goes back to the work of Pr\"ufer for abelian groups \cite{Pruefer}. The theory of purity gives a nice framework to understand the structure of possibly large modules. It also appears in the context of logic and model theory
of modules \cite{Prest}. An important result is the correspondence between open sets in the Ziegler spectrum, definable subcategories of $\Mod A$ and Serre subcategories of the abelian category of finitely presented functors $\Mod A \to \Ab$, see \cite{Crawley-Boevey2, Krause, Herzog}. This provides the main connection with constructible subcategories: An open set in the Ziegler spectrum is compact if and only if the corresponding definable subcategory is a constructible subcategory. The following describes the key steps of the proof of Theorem A.
\begin{itemize}
    \item \textbf{The Ziegler spectrum and schemes of modules.} For $n\in \mathbb{N}$ let $\vMod(A,n)$ denote the scheme of $n$-dimensional $A$-modules. (One may also think in terms of module varieties if the field is algebraically closed, see for example \cite{CBgeo}.) The $k$-points in $\vMod(A,n)$ correspond to $n$-dimensional $A$-modules. For a subcategory $\mathcal{X} \subseteq \Mod A$ let $\vMod(\mathcal{X},n)$ be the points in $\vMod(A,n)$ such that the corresponding $A$-modules are contained in $\mathcal{X}$. We establish a connection between the Ziegler spectrum of $A$ and the scheme  $\vMod(A,n)$, which shows that if $\mathcal{X}$ is a constructible subcategory, then $\vMod(\mathcal{X}, n) \subseteq \vMod(A,n)$ is a constructible set, that is, a finite union of locally closed sets in the Zariski topology.
    \item \textbf{Curves passing through distinct orbits.} To show Theorem A we may assume that the constructible subcategory $\mathcal{X} \subseteq \Mod A$ contains infinitely many non-isomorphic modules of some fixed dimension $n\in \mathbb{N}$ and deduce that $\mathcal{X}$ is of unbounded type. In this case $\vMod(\mathcal{X},n)$ contains infinitely many $k$-points in distinct orbits under the natural action of $GL_n(k)$. Since additionally $\vMod(\mathcal{X}, n) \subseteq \vMod(A,n)$ is a constructible set, by a result shown by Andres Fernandez Herrero in the appendix  there is a curve inside $\vMod(\mathcal{X},n)$ with the following property: There are infinitely many non-isomorphic indecomposable $A$-modules that appear as direct summands of the $A$-modules corresponding to points of the curve.
    \item \textbf{Lift to a localisation of the coordinate ring.} Let $B$ denote a finite localisation of the coordinate ring of the curve inside $\vMod(\mathcal{X},n)$ such that $B$ is regular. Then $B$ is a Dedekind ring and its Ziegler spectrum is well-understood. As a consequence, we show a dichotomy for constructible subcategories $\mathcal{Y} \subseteq \Mod B$. The constructible subcategory $\mathcal{X} \subseteq \Mod A$ lifts to a constructible subcategory $\mathcal{Y} \subseteq \Mod B$, which is large in terms of the dichotemy. Thus, $\mathcal{Y}$ contains $B/m^i$ for all but finitely many maximal ideals $m \subseteq B$ and all $i>0$. 
    \item \textbf{Large modules.} The $B$-modules $B/m^i \in \mathcal{Y}$ correspond to $A$-modules $X_{m,i} \in \mathcal{X}$, which are the candidates to show that $\mathcal{X}$ is of unbounded type. Using the Ziegler spectrum of $A$ their behaviour can be controlled by the infinite dimensional $A$-module $X$ corresponding to $\Frac(B)$. Indeed, by the properties of the curve it can be shown that $X$ has an infinite dimensional indecomposable direct summand. Consequently, it follows that for every $m$ there are infinitely many non-isomorphic indecomposable direct summands of the modules $X_{m,i}$ with $i >0$. The last step is to confirm that their dimension is not bounded. This is deduced from the existence of the short exact sequences $0 \to X_{m,i} \to X_{m,j} \to X_{m,j-i} \to 0$ for $j>i>0$.
\end{itemize}

    The second Brauer-Thrall conjecture states that for a finite dimensional algebra over an infinite field, infinite representation type implies strongly unbounded representation type. For algebraically closed fields it was proven by Bautista for $\textnormal{char}(k) \neq 2$ and the case $\textnormal{char}(k) = 2$ was completed by Bongartz \cite{Bautista, Bongartz}. This also implies the conjecture for perfect infinite fields, but it remains open in general. Under additional assumptions on the algebra, a variant of the second Brauer-Thrall conjecture for constructible subcategories  is shown.\\
    \\
    \textbf{Theorem B.} (\Cref{kgbtt}) \textit{Let $k$ be a perfect infinite field, $\bar{k}$ the algebraic closure of $k$, $A$ a finite dimensional $k$-algebra such that the Krull-Gabriel dimension of $A\otimes_k \bar{k}$ is defined and $\mathcal{X} \subseteq \Mod A$ a constructible subcategory. If $\mathcal{X}$ is of infinite type, then $\mathcal{X}$ is of strongly unbounded type.}\\

    The statement in Theorem B is wrong for arbitrary finite dimensional algebras, see \Cref{counter}. The Krull-Gabriel dimension, introduced by Geigle \cite{Geigle}, is a measure for the complexity of the module category of $A$. A conjecture of Prest states that the Krull-Gabriel dimension of $A$ is defined if and only if $A$ is domestic \cite[Conjecture 9.1.15]{Prest}. The conjecture is known for several classes of algebras, for a detailed discussion see \cite[Section 1]{Pastuszak}.

    Lastly, we show how exact structures, in the sense of Quillen \cite{Quillen}, and matrix reductions, in the framework of lift categories \cite{CB}, yield constructible subcategories. An exact structure $\mathcal{E}$ is a collection of kernel-cokernel pairs subject to some closure properties. Further, $\mathcal{E}$ is \emph{finitely generated} if it equals the smallest exact structure containing a specific kernel-cokernel pair. For a finite dimensional algebra $A$ and a finitely generated exact structure $\mathcal{E}$ on $\mod A$ the modules $X\in \Mod A$ that behave injectively with respect to $\mathcal{E}$ ($\coker \Hom_{A}(f, X) = 0$ for all $(f,g)\in \mathcal{E}$) form a constructible subcategory of $\Mod A$. It follows from the connection between exact structures and the theory of purity shown in \cite{Sch} that, in this way, one can essentially describe every constructible subcategory.

    Matrix reductions in terms of lift categories were introduced by Crawley-Boevey. A main feature is the extent of generality in which they work. In particular, for finite dimensional algebras there is no restriction on the field. A connection between the second Brauer-Thrall conjecture and the existence of certain infinite dimensional modules, originally proven via lift categories, see \cite[Theorem 9.6]{Crawley-Boevey}, can partially be proven using the Ziegler spectrum \cite[Theorem 9.6]{Herzog}. This hints at a relation between matrix reductions and the theory of purity, which is established in this work. As a consequence, we show that performing matrix reductions for a finite dimensional algebra $A$ gives rise to constructible subcategories of $\Mod A$. By the discussion before, these constructible subcategories essentially come from exact structures on $\mod A$.\\
    \\
    \textbf{Theorem C} (\Cref{translation}) \textit{Let $k$ be a field and $A$ a finite dimensional $k$-algebra. Matrix reductions for $A$ in the form of lift categories translate to reductions of exact structures. Moreover, the exact structures involved are finitely generated.}\\

    The concept of reductions of exact structures was introduced in \cite{BHLR} and it is described by switching between differenct exact structures on $\mod A$. The authors introduce the notion of length and Gabriel-Roiter measure relative to an exact structure and investigate how these change when performing reductions of exact structures. Further, it was shown in \cite[Section 4]{BHLR} that matrix reductions for a given example translate to reductions of exact structures. Whether or not this is always possible was an open question. Theorem C yields a positive answer.

\section{Preliminaries}

A main tool in this work is the theory of purity for definable categories in the sense of Prest \cite{Prest2}. While we are mainly concerned with module categories over rings, this very general setup gives a more complete point of view and  simplifies some arguments.   

\subsection{Definable categories.} For an additive category $\mathcal{A}$ with filtered colimits and products let $C(\mathcal{A})$ be the category of additive functors $\mathcal{A} \rightarrow \Ab$ that commute with filtered colimits and products. The category $C(\mathcal{A})$ is abelian and \mbox{(co{-)}kernels} are computed locally. Let $\Ex(C(\mathcal{A}), \Ab)$ be the category of exact functors $C(\mathcal{A}) \rightarrow \Ab$. We call $\mathcal{A}$ \emph{definable} if $C(\mathcal{A})$ is essentially small and $\mathcal{A} \rightarrow \Ex(C(\mathcal{A}), \Ab), X\mapsto \bar{X}$ is an equivalence, where $\bar{X}(F) = F(X)$ for $X\in \mathcal{A}$ and $F\in C(\mathcal{A})$. For example, finitely accessible categories with products are definable. The definition of definable categories here is equivalent to the one in \cite[Chapter 10]{Prest2} as well as Krause's definition of exactly definable categories in  \cite{Krause}, see \cite[Theorem 12.10]{Prest2}.

\begin{prop*} \cite[Corollary 2.9]{Krause} There is a one to one correspondence between definable categories $\mathcal{A}$ and essentially small abelian categories $\mathcal{C}$ up to equivalence, given by 
\begin{align*}
    \mathcal{A} \mapsto C(\mathcal{A}) \quad \text{ and } \quad \mathcal{C} \mapsto \Ex(\mathcal{C}, \Ab). 
\end{align*}
\end{prop*}

\subsection{Purity} Let $\mathcal{A}$ be a definable category. A sequence $0 \rightarrow X \rightarrow Y \rightarrow Z \rightarrow 0$ in $\mathcal{A}$ is \emph{pure-exact} if 
\begin{align*}
    0 \longrightarrow F(X) \longrightarrow F(Y) \longrightarrow F(Z) \longrightarrow 0
\end{align*}
is exact for all $F\in C(\mathcal{A})$. In that case $X$ is a \emph{pure subobject} of $Y$. An object $X\in \mathcal{A}$ is \emph{pure-injective} if every pure-exact sequence starting in $X$ splits. The isomorphism classes of indecomposable pure-injective objects in $\mathcal{A}$ form a topological space, the \emph{Ziegler spectrum} of
$\mathcal{A}$, denoted by $\Ind \mathcal{A}$. A subset $\mathcal{U}$ of $\Ind \mathcal{A}$ is closed if there is a definable subcategory $\mathcal{X}$ of $\mathcal{A}$ such that $\mathcal{U} = \mathcal{X} \cap \Ind \mathcal{A}$.

A full additive subcategory $\mathcal{X}$ of $\mathcal{A}$ is a \emph{definable subcategory} if $\mathcal{X}$ is closed under filtered colimits, products and pure subobjects. Equivalently, there is a set of functors $S \subseteq C(\mathcal{A})$ such that
\begin{align*}
    \mathcal{X} = \mathcal{X}_S = \{X\in \mathcal{A} \mid F(X) = 0 \text{ for all }F\in {S}\}.
\end{align*}
We write $\mathcal{X}_F = \mathcal{X}_{\{F\}}$ for $F \in C(\mathcal{A})$.

A full additive subcategory of an abelian category is a \emph{Serre subcategory} if it is closed under extensions, subobjects and quotients. The following assignments are important for the theory of purity.

\begin{thm*} \cite[Corollary 4.7, Theorem 6.2]{Krause} Let $\mathcal{A}$ be a definable category. There are one to one correspondences between closed sets $\mathcal{U}$ in $\Ind \mathcal{A}$, Serre subcategories $\mathcal{S}$ of $C(\mathcal{A})$ and definable subcategories $\mathcal{X}$ of $\mathcal{A}$, given by
\begin{align*}
    \mathcal{U} &\mapsto \mathcal{S}_\mathcal{U}=\{F\in C(\mathcal{A}) \mid F(X) = 0 \text{ for all }X\in \mathcal{U}\},\\
    \mathcal{S}&\mapsto \mathcal{X}_{\mathcal{S}} = \{X \in \mathcal{A} \mid F(X) = 0\text{ for all }F\in \mathcal{S} \},\\
    \mathcal{X} &\mapsto \mathcal{U}_\mathcal{X} = \mathcal{X}\cap \Ind \mathcal{A}.
\end{align*}
\end{thm*}

A definable subcategory $\mathcal{X}$ of a definable category $\mathcal{A}$ is again a definable category \cite[Proposition 4.1]{Krause}. The functor $C(\mathcal{A}) \rightarrow C(\mathcal{X}), F \mapsto F|_\mathcal{X}$ induces an equivalence between the Serre quotient $C(\mathcal{A})/\mathcal{S}_\mathcal{X}$ and $C(\mathcal{X})$. Moreover, the inclusion $\mathcal{X} \subseteq \mathcal{A}$ induces a homeomorphism $\Ind \mathcal{X} \rightarrow \mathcal{U}_\mathcal{X}$, see \cite[Corollary 6.3]{Krause}.

\begin{prop*} \cite[Theorem 14.1]{Prest2} Let $\mathcal{U}$ be a closed set in $\Ind \mathcal{A}$. Then $\Ind \mathcal{A} \setminus \mathcal{U}$ is compact if and only if there is $F\in C(\mathcal{A})$ with $\mathcal{U} =  \mathcal{X}_{F} \cap \Ind \mathcal{A}$.
\end{prop*}

\subsection{Definable functors} An additive functor between definable categories that commutes with filtered colimits and products is called \emph{definable}. This definition agrees with the one in \cite[Section 7]{Krause} by \cite[Theorem 13.1]{Prest2}. The correspondence in Proposition 1.1 is functorial in the following sense. It induces a duality between the category of essentially small abelian categories with exact functors as morphisms, and the category of definable categories with definable functors as morphisms, see again \cite[Theorem 13.1]{Prest2}. 

\begin{thm*} \cite[Theorem 7.8]{Krause} Let $f\colon \mathcal{A}\rightarrow \mathcal{B}$ be a definable functor between definable categories $\mathcal{A}$ and $\mathcal{B}$. \begin{itemize}
    \item[\rm (1)] The functor $f$ preserves pure-injectivity and pure-exact sequences. 
    \item[\rm (2)] For $U \subseteq \Ind \mathcal{A}$ let $V \subseteq \Ind \mathcal{B}$ be the set of all indecomposable direct summands of objects in $f(U)$. The closure $\overline{V}$ consists of all indecomposable direct summands of objects in $f(\overline{U})$.
\end{itemize}
\end{thm*}

\subsection{Finitely accessible categories} Natural examples of definable categories are finitely accessible categories with products, see \cite[Proposition 2.4]{Krause}. An additive category $\mathcal{A}$ with filtered colomits is \emph{finitely accessible} if the full subcategory of finitely presented objects in $\mathcal{A}$ is essentially small, and every object in $\mathcal{A}$ is a filtered colimit of finitely presented objects. An object $X\in \mathcal{A}$ is \emph{finitely presented} if $\Hom_{\mathcal{A}}(X,-)$ commutes with filtered colimits. 

If $\mathcal{A}$ is a finitely accessible category with products, then $C(\mathcal{A})$ can (and will) be identified with the category of finitely presented functors $\mathcal{A} \rightarrow \Ab$ \cite[Section 2]{Krause}. A functor $F \colon \mathcal{A} \rightarrow \Ab$ is \emph{finitely presented} if there is a short exact sequence
\begin{align*}
    \Hom_{\mathcal{A}}(Y,-) \longrightarrow \Hom_{\mathcal{A}}(X,-) \longrightarrow F \longrightarrow 0
\end{align*}
with $X,Y \in \mathcal{A}$ finitely presented.

\subsection{Module categories} Let $R$ be a ring (associative with identiy). The category $\Mod R$ of left $R$-modules is a finitely accessible category with products and thus definable. The finitely presented objects in $\Mod R$ form the full subcategory $\mod R$ of finitely presented modules. We set $C(R) = C(\Mod R)$ and $\Ind R = \Ind \Mod R$. The following are important properties of the Ziegler spectrum in this setting.
\begin{itemize} 
    \item[(1)] The space $\Ind R$ is compact \cite{Ziegler}. This can be deduced from Proposition 1.2 by considering the forgetful functor $F\colon \Mod R \rightarrow \Ab$.
    \item[(2)] For $X\in \Mod R$ the \emph{endolength} of $X$, denoted by $\lend(X)$, is the length of $X$ as an $\textnormal{End}_R(X)$-module. If $X$ is indecomposable and $\lend(X) < \infty$, then $\{X\}$ is a closed set in $\Ind R$, see \cite[Theorem 6.14]{Krause0}. 
    \item[(3)] For all $n \in \mathbb{N}$ the set $\{X \in \Mod R \mid X \text{ is indecomposable and } \lend(X) \leq n\}$ is closed in $\Ind R$, see \cite[Theorem 9.3]{Herzog}.
    \item[(4)] Let $R = A$ be a finite dimensional algebra and $X\in \Mod A$ indecomposable. Then $\{X\}$ is open in $\Ind A$ if and only if $\dim X < \infty$, see \cite[Example 3.7]{Krause0}. Note that $\dim X < \infty$ implies $\lend(X) < \infty$, so then $\{X\}$ is also closed.
\end{itemize}

\section{Constructible Subcategories}

Let $\mathcal{A}$ be a definable category. A full subcategory $\mathcal{X}$ of $\mathcal{A}$ is a \emph{constructible subcategory} if there is $F \in C(\mathcal{A})$ such that
\begin{align*}
    \mathcal{X} = \mathcal{X}_F = \{X \in \mathcal{A} \mid F(X) = 0\}.
\end{align*}
In that case $\mathcal{X}$ is also a definable subcategory. Our analysis mainly deals with the case $\mathcal{A} = \Mod A$ for a finitely generated $k$-algebra $A$, where $k$ is a field. 

\begin{rem}\label{int} \rm Let $\mathcal{A}$ be a definable category.
\begin{itemize}
    \item[(1)] A definable subcategory $\mathcal{X}$ of $\mathcal{A}$ is constructible if and only if for the closed set $\mathcal{U} \subseteq \Ind \mathcal{A}$ corresponding to $\mathcal{X}$ its complement $\Ind \mathcal{A} \setminus \mathcal{U}$ is compact, see Theorem 1.2 and Proposition 1.2.
    \item[(2)] Let $\mathcal{X}, \mathcal{Y}$ be constructible subcategories of $\mathcal{A}$. The intersection $\mathcal{X} \cap \mathcal{Y}$ is again a constructible subcategory of $\mathcal{A}$. Indeed, if $\mathcal{X} = \mathcal{X}_F$ and $\mathcal{Y} = \mathcal{X}_G$ for $F, G \in C(\mathcal{A})$, then $\mathcal{X} \cap \mathcal{Y} = \mathcal{X}_{F\oplus G}$ with $F\oplus G \in C(\mathcal{A})$. 
\end{itemize}
\end{rem}

We show that the property of being a constructible subcategory as well as a definable subcategory is transitive.

\begin{lem}\label{trans} Let $\mathcal{A}$ be a definable category. If $\mathcal{X}$ is a definable (constructible) subcategory of $\mathcal{A}$ and $\mathcal{Y}$ a definable (constructible) subcategory of $\mathcal{X}$, then $\mathcal{Y}$ is a definable (constructible) subcategory of $\mathcal{A}$. 
\end{lem}

\begin{proof} Let $\mathcal{X}= \mathcal{X}_S$ and $\mathcal{Y} = \mathcal{X}_T$ for $S \subseteq C(\mathcal{A})$ and $T \subseteq C(\mathcal{X})$, see Section 1.2. The equivalence $C(\mathcal{X}) \simeq C(\mathcal{A})/ \mathcal{S}_{\mathcal{X}}$ shows that every $G\in T$ factors as 
\begin{equation*}
\begin{tikzcd}
    \mathcal{X} \arrow[d, hook] \arrow[r, "G"] & \Ab\\
    \mathcal{A} \arrow[ru, "H_G", swap] & 
\end{tikzcd} 
\end{equation*}
for some $H_G \in C(\mathcal{A})$. Now $\mathcal{Y} = \mathcal{X}_{\tilde{S}}$ with $\tilde{S} = \{ F\oplus H_G \mid F\in S, G \in T \} \subseteq C(\mathcal{A})$. Clearly, if $S$ and $T$ are singletons, then so is $\tilde{S}$.
\end{proof}

The following are typical examples of constructible subcategories of a module category that itself are equivalent to a module category.

\begin{exa}\label{modex}\rm Let $R$ be a ring.
\begin{itemize}
    \item[(1)] For an ideal $I \subseteq R$ the category $\Mod R/I$ is equivalent to the full subcategory of all $X \in \Mod R$ with $IX = 0$. This is a definable subcategory by checking the closure properties. The condition $I X = 0$ can be expressed as $F(X) = 0$, where $F\colon \Mod R \rightarrow \Ab$ is defined by the exact sequence
    \begin{align*}
       \Hom_{R}(R/I,-) \longrightarrow \Hom_{R}(R,-) \longrightarrow F \longrightarrow 0. 
    \end{align*}
    If $I$ is finitely generated, then $R/I \in \mod R$ and $F\in C(R)$. Thus, in this case $\Mod R/I$ is equivalent to the constructible subcategory $\mathcal{X}_F$ of $\Mod R$.
    \item[(2)] For a morphism $P_1 \rightarrow P_2$ between projective modules in $\mod R$ consider the finitely presented functors $F,G\in C(R)$ defined by the exact sequence
    \begin{align*}
       0 \longrightarrow F \longrightarrow  \Hom_{R}(P_2, -) \longrightarrow \Hom_{R}(P_1,-) \longrightarrow G \longrightarrow0.
    \end{align*}
    Then $X\in \mathcal{X}_{F\oplus G}$ if and only if $\Hom_{R}(P_2, X) \rightarrow \Hom_{R}(P_1,X)$ is a bijection. This is a universal localisation in the sense of Schofield \cite{Schofield}, so there is a ring epimorphism $R \rightarrow S$ such that $\Mod S \simeq \mathcal{X}_{F\oplus G}$.
\end{itemize}
\end{exa}

From now on we fix a field $k$. While most of our methods work for constructible subcategories of the module category of a finitely generated $k$-algebra, in all of the examples we deal with finitely presented $k$-algebras. The case of a finitely presented $k$-algebra is particularly nice.

\begin{prop}\label{embed} Let $A$ be a finitely generated (presented) $k$-algebra. Then $\Mod A$ is equivalent to a definable (constructible) subcategory of $\Mod \tilde{A}$, where $\tilde{A}$ is a finite dimensional $k$-algebra. Further, the equivalence changes the dimension of modules by a factor of $2$.
\end{prop}

\begin{proof} Let $A \cong k \langle x_1, \dots, x_n \rangle /I$. The category $\Mod A$ is equivalent to a definable subcategory $\mathcal{X}$ of $\Mod k \langle x_1, \dots, x_n \rangle$, respectively constructible subcategory if $I$ is finitely generated, by \Cref{modex} (1). Consider the $(n+1)$-Kronecker quiver
    \begin{equation*}
        Q = \begin{tikzcd}
        1 \arrow[r, "\alpha_1", shift left=2.5] \arrow[r, "\cdots", shift right=1] \arrow[r, "\alpha_{n+1}", shift right=2.5, swap] & 2
        \end{tikzcd}
    \end{equation*}
and set $\tilde{A} = k Q$, which is a finite dimensional $k$-algebra. We identify $\tilde{A}$-modules with quiver representations.

There is a fully faithful functor $f\colon \Mod k \langle x_1, \dots, x_n \rangle \rightarrow \Mod \tilde{A}$ that maps a module $X \in \Mod A$ to the quiver representation, where $\alpha_i$ corresponds to the map $X \rightarrow X, y \mapsto x_i y$ for $1 \leq i \leq n$ and $\alpha_{n+1}$ corresponds to $1_X$. The essential image of $f$ equals the collection of all quiver representations, where the map corresponding to $\alpha_{n+1}$ is a bijection. This condition can be expressed via \Cref{modex} (2) using the morphism $kQe_2 \rightarrow kQe_1, x \mapsto x \alpha_{n+1}$. It follows that the essential image of $f$ is a constructible subcategory of $\Mod \tilde{A}$. By \Cref{trans} the essential image of $f|_\mathcal{X}$, which is equivalent to $\Mod A$, is a definable subcategory of $\Mod \tilde{A}$, respectively constructible subcategory if $I$ is finitely generated. Further, by construction $\dim f(X) = 2 \dim X$ for $X\in \Mod A$.
\end{proof}

The above result shows that a constructible subcategory of the module category of a finitely presented $k$-algebra is equivalent to a constructible subcategory of the module category of a finite dimensional $k$-algebra. Note that a finite dimensional $k$-algebra is in particular finitely presented.

For a constructible subcategory $\mathcal{X}$ of $\Mod A$, where $A$ is a $k$-algebra, we are mainly interested in the full subcategory $\fin \mathcal{X}$ of finite dimensional modules in $\mathcal{X}$. Because the Hom-spaces are finite dimensional over $k$, it follows that $\fin \mathcal{X}$ is a Krull-Schmidt category. It may happen that $\fin \mathcal{X}$ is trivial, while $\mathcal{X}$ is not. For example, for the Weyl algebra $A = k \langle x,y \rangle/(xy-yx - 1)$ there are no non-zero finite dimensional $A$-modules.

\begin{lem}\label{duality} Let $A$ be a $k$-algebra and $\mathcal{X}$ a constructible subcategory of $\Mod A$. There is a constructible subcategory $\mathcal{Y}$ of $\Mod A ^\op$ such that $(\fin \mathcal{X})^\op \simeq \fin \mathcal{Y}$. 
\end{lem}
\begin{proof} Let $F \in C(A)$ such that $\mathcal{X} = \mathcal{X}_F$. The standard $k$-duality $D = \Hom_{k}(-,k)$ induces a duality between $\fin \Mod A$ and $\fin \Mod A^\op$. Consider the second duality {$d\colon C(A) \rightarrow C(A^\op)$} defined by $\Hom_{A}(X,-) \mapsto (-)\otimes_A X$ for $X\in \mod A$, see for example \cite[Theorem 10.3.4]{Prest}. The $k$-linear bijection
\begin{align*}
    \Hom_{A}(X,DY) \cong D(Y \otimes_A X), \quad f\mapsto (y \otimes x\mapsto f(x)(y))
\end{align*}
shows that $G(DY) = 0$ if and only if $dG(Y) = 0$ for $G\in C(A)$ and $Y \in \Mod A^\op$. Indeed, this is immediate for $G = \Hom_{A}(X,-)$ and in general consider a projective resolution of $G$ in $C(A)$. It follows that 
$(\fin \mathcal{X})^\op \simeq \fin \mathcal{Y}$ via $D$ for $\mathcal{Y} = \mathcal{X}_{dF}$.
\end{proof}

We provide examples for the full subcategory of finite dimensional modules in a constructible subcategory of the module category of a finite dimensional $k$-algebra. By \Cref{int} (2) every finite intersection of them is again such a category.

\begin{exa}\label{exa} \rm Let $A$ be a finite dimensional $k$-algebra. The following are categories of the form $\fin \mathcal{X}$, where $\mathcal{X}$ is a constructible subcategory of $\Mod A$. 
\begin{itemize} 
\item[(1)] Every full additive subcategory of $\mod A$ containing precisely finitely many isomorphism classes of indecomposable modules.
\item[(2)] Every full additive subcategory of $\mod A$ containing all but finitely many isomorphism classes of indecomposable modules. 
\item[(3)] Hom- and Ext$^n$-orthogonals of a finite dimensional module within $\mod A$ with $n \in \mathbb{N}$.
    \item[(4)] The category of all modules in $\mod A$ with projective (injective) dimension bounded by some $n \in \mathbb{N}$. 
    \item[(5)] The category of all modules in $\mod A$ that are generated (cogenerated) by a module in $\mod A$.
    \item[(6)] Every functorially finite torsion(-free) class in $\mod A$. 
    \item[(7)] The $\Delta$-good module category $\mathcal{F}(\Delta)$, which consists of all modules in $\mod A$ filtered my standard modules, if $A$ is a quasi-hereditary algebra.
\end{itemize}
\end{exa}
\begin{proof} Throughout we make use of duality, see \Cref{duality}.

(1), (2) Consider finitely many indecomposable modules $X_1, \dots, X_n$ in $\mod A$. The set  $\mathcal{U} = \{X_1, \dots, X_n\}$ as well as $\Ind \mathcal{A} \setminus \mathcal{U}$ is closed in $\Ind A$ by Section 1.5 (4). Hence, both sets are compact since so is $\Ind A$, see {Section 1.5 (1)}. It follows that $\mathcal{U} = \mathcal{X}_F \cap \Ind \mathcal{A}$ and $\Ind \mathcal{A} \setminus \mathcal{U} = \mathcal{X}_G \cap \Ind \mathcal{A}$ for $F,G \in C(A)$ by Proposition 1.2. Now $\fin \mathcal{X}_F$ (respectively $\fin \mathcal{X}_G$) is the full additive subcategory of $\mod A$ containing precisely the  indecomposable modules $X_1, \dots, X_n$ up to isomorphism (respectively all but $X_1, \dots, X_n$).

(3) It is well-known that the functors $\Hom_A(M,-)$ and $\Ext_A^n(M,-)$ are finitely presented for $M\in \mod A$, so contained in $C(A)$. The case of $\Hom_A(-,M)$ and $\Ext_A^n(-,M)$ follows by duality.

(4) The inequality $\text{pdim}\,X \leq n$, respectively $\text{idim}\, X \leq n$, holds if and only if $\Ext_{A}^{n+1} (X,S) = 0$, respectively $\Ext_A^{n+1} (S,X) = 0$, where $S$ is the direct sum of the finitely many simple modules in $\mod A$. Thus, we can apply (3). 

(5) Recall that for $M\in \mod A$ the category $\gen M$ of all modules in $\mod A$ that are generated by $M$ equals the collection of all $X\in \mod A$ that admit an epimorphism $M^n \rightarrow X$ for some $n\in \mathbb{N}$. For $X\in \Mod A$ let
\begin{align*}
    \textnormal{ev}_X \colon M \otimes_k \Hom_A(M,X) \longrightarrow X,\quad x \otimes f \mapsto f(x). 
\end{align*}
The functor $F\colon \Mod A \rightarrow \Ab, X\mapsto \coker \textnormal{ev}_X$ commutes with filtered colimits and products, since so do $M\otimes_k (-), \Hom_{A}(M,-)$ and $\coker (-)$. Hence, $F\in C(A)$. Further, for $X\in \mod A$ we have $X\in \gen M$ if and only if $ F(X) = 0$. It follows that $\fin \mathcal{X}_F = \gen M$. By duality, the category $\cogen M$ of all modules in $\mod A$ cogenerated by $M$ equals $\fin \mathcal{X}_G$ for some $G\in C(A)$.

(6) Let $(\mathcal{T}, \mathcal{F})$ be a torsion pair in $\mod A$. Then $\mathcal{T} = \gen M$ for $M\in \mod A$ if and only if $\mathcal{T}$ is functorially finite if and only if $\mathcal{F}$ is functorially finite if and only if  $\mathcal{F} = \cogen N$ for $N \in \mod A$ \cite{Smalo}. Thus, the claim follows by (5).  

(7) A precise definition of $\mathcal{F}(\Delta)$ is given in \cite[Part II]{Ringel1}. Let $\nabla_1, \dots, \nabla_n$ be a complete list of costandard modules. Then
\begin{align*}
    \mathcal{F}(\Delta) = \{X\in \mod A \mid \Ext_A^1(X, \bigoplus_{i=1}^n \nabla_i ) = 0\}
\end{align*}
by \cite[Theorem 4*]{Ringel1} and we can apply (2).
\end{proof}

Let $k$ be a field, $A$ a $k$-algebra and $\mathcal{X}$ a constructible subcategory of $\Mod A$. We say that $\mathcal{X}$ is of
\begin{itemize}
    \item[(1)] \emph{infinite type} if there are infinitely many non-isomorphic finite dimensional indecomposable modules in $\fin \mathcal{X}$,
    \item[(2)] \emph{unbounded type} if there is no bound on the dimension of indecomposable modules in $\fin \mathcal{X}$, and
    \item[(3)] \emph{strongly unbounded type} if there is a sequence  $n_1 < n_2 < \dots $ of natural numbers such that for all $i \in \mathbb{N}$ there are infinitely many non-isomorphic indecomposable modules of dimension $n_i$ in $\fin \mathcal{X}$.
\end{itemize}
Cleary, strongly unbounded type implies unbounded type, which implies infinite type. We are concerned to what extend the non-trivial implications hold.

\section{Brauer-Thrall I}

The first Brauer-Thrall conjecture states that infinite type implies unbounded type for the module category of a finite dimensional algebra. It was first proven by Roiter \cite{Roiter}. In \cite{BG} it was shown that the statement can be extended to finitely generated algebras $A$ over a field $k$. We prove a variant of the first Brauer-Thrall conjecture for constructible subcategories $\mathcal{X}$ of $\Mod A$. That is, if $\mathcal{X}$ is of infinite type, then $\mathcal{X}$ is of unbounded type. This is done via the scheme of $n$-dimensional $A$-modules \cite{Morrison}, which we describe next, and its connection to the Ziegler spectrum. For an overview of the key steps of the proof, see the introduction.

Let $A \cong k\langle x_1, \dots, x_m\rangle /I$. An $A$-module of dimension $n\in \mathbb{N}$ is given by $n\times n$ matrices $M_1, \dots, M_m$ over $k$ fulfilling the relations $I$. This induces polynomial relations on the entries of the matrices, so they form an affine $k$-scheme with coordinate ring $C$, which is a finitely generated commutative $k$-algebra. We call 
\begin{align*}
\vMod(A,n) = \Spec(C) = \{p \subseteq C \mid p \text{ is a prime ideal}\}    
\end{align*}
the \emph{scheme of $n$-dimensional $A$-modules}. A maximal ideal $m \subseteq C$ with $C/m \cong k$ is called a \emph{$k$-point}.

\begin{rem}\label{geom}\rm The following is known, see \cite[Remark 4.7]{Ringel0}. Let $n\in \mathbb{N}$ and $C$ be the coordinate ring of $\text{Mod}(A,n)$. There is an $A$-$C$-bimodule $M$ free over $C$ of rank $n$ such that every $n$-dimensional $A$-module is isomorphic to some $M \otimes_C C/m$ for a $k$-point $m\in \vMod(A,n)$. In general, an arbitrary point $p\in \vMod(A,n)$ corresponds to the $A$-module $M \otimes_k \Frac(C/p)$. We illustrate the construction of $M$ (and $C$).

Let $A \cong k \langle x_1, \dots, x_m \rangle/I$ and $k[t]$ be the polynomial ring over $k$ in $m n^2$-many variables. Consider $n\times n$ matrices $M_1, \dots, M_m$ over $k[t]$ such that each entry is a different indeterminate of $k[t]$. Let $J \subseteq k[t]$ be the smallest ideal such that $p(M_1, \dots, M_m) = 0$ over $k[t]/J$ for all $p \in I$. Then the matrices $M_1, \dots, M_m$ over $C = k[t]/J$ determine an $A$-$C$-bimodule $M$, which is free over $C$ of rank $n$. An $n$-dimensional $A$-module $X$ is given by $n\times n$ matrices $N_{1}, \dots, N_{m}$ over $k$. These matrices are obtained by evaluating $M_1, \dots, M_m$ at a suitable point in $k^{mn^2}$, which corresponds to a maximal ideal $m\subseteq C$ such that $C/m \cong k$ and $M \otimes_C C/m \cong X$.
\end{rem}

In what follows fix $n \in \mathbb{N}$, the coordinate ring $C$ of the scheme of $n$-dimensional $A$-modules $\vMod(A,n)$, and the $A$-$C$-bimodule $M$ as in \Cref{geom}.

The \emph{Zariski topology} on $\vMod(A,n)$ has the sets $D_a = \{p\in \vMod(A,n) \mid a \notin p\}$ with $a\in C$ as a basis of open sets. A different topology on $\vMod(A,n)$ is more important for our purposes. The \emph{constructible topology} on $\vMod(A,n)$ is the coarsest topology such that $D_a$ is both open and closed for all $a\in C$, see \cite[Tag 08YF]{sp} for some details. Thus, also all finite combinations of unions and intersections of the sets $D_a$ as well as $\vMod(A,n) \setminus D_a$ are open and closed. Such a set is \emph{constructible} in $\vMod(A,n)$ and the constructible sets coincide with the compact open sets in the constructible topology of $\vMod(A,n)$, see \cite[Tag 0905]{sp}. There is a nice connection with the Ziegler spectrum of $C$. 

\begin{thm}\label{constr}\cite[p. 555]{Herzog} The map
\begin{align*}
    \vMod(A,n) \longrightarrow \Ind C, \quad p \mapsto \Frac(C/p)
\end{align*}
induces a homeomorphism between $\vMod(A,n)$ with the constructible topology and the closed set $\{X \in \Ind C \mid \lend (X) = 1\} \subseteq \Ind C$ with the subspace topology.
\end{thm} 

For a full additive subcategory $\mathcal{X} \subseteq \Mod A$ let $\vMod(\mathcal{X},n)$ be the set of all $p \in \vMod(A,n)$ with $M \otimes_C \Frac(C/p) \in \mathcal{X}$. In particular, if $X$ is an $n$-dimensional $A$-module corresponding to a $k$-point $m \in \vMod(A,n)$, then $X \in \mathcal{X}$ if and only if \mbox{$m \in \vMod(\mathcal{X}, n)$}, see \Cref{geom}.

\begin{lem}\label{constru} Let $\mathcal{X}$ be a full additive subcategory of $\Mod A$.\begin{itemize}
    \item[(1)] If $\mathcal{X}$ is a definable subcategory, then $\vMod(\mathcal{X},n)$ is closed in the constructible topology of $\vMod(A,n)$.
    \item[(2)] If $\mathcal{X}$ is a constructible subcategory, then $\vMod(\mathcal{X},n)$ is a constructible set in $\vMod(A,n)$. 
\end{itemize}
\end{lem}
\begin{proof} If $\mathcal{X}$ is definable, then $\mathcal{X} = \mathcal{X}_S$ for a set of functors $S\subseteq C(A)$, see Section 1.2. The functor $f = M \otimes_C (-) \colon \Mod C \to \Mod A$ commutes with filtered colimits and products since $M$ is free over $C$ of rank $n$. Thus, $T = \{F\circ f \mid F \in S\} \subseteq C(C)$. Now $\mathcal{X}_T \subseteq \Mod C$ equals the collection of all $Y \in \Mod C$ with $M \otimes_C Y \in \mathcal{X}$. Consider the corresponding closed set $\mathcal{U} = \mathcal{X}_T \cap \Ind C$ and $\mathcal{V} = \{X\in \Ind C \mid \lend(X) = 1\}$, which is also closed in $\Ind C$ by Section 1.5 (3). By definition $\vMod(\mathcal{X},n)$ equals the preimage of $\mathcal{U} \cap \mathcal{V}$ under the homeomorphism in \Cref{constr}. It follows that $\vMod(\mathcal{X}, n)$ is closed.

If $\mathcal{X}$ is constructible, then $S$ and $T$ are singletons and so $\Ind C \setminus \mathcal{U}$ is compact by Proposition 1.2. Hence, $\mathcal{V} \setminus \mathcal{U} = \mathcal{V} \cap (\Ind C\setminus \mathcal{U})$ is compact as the intersection of a closed and compact set. It follows that $\vMod(A,n) \setminus \vMod(\mathcal{X},n)$, which is the preimage of $\mathcal{V} \setminus \mathcal{U}$ under the homeomorphism in \Cref{constr}, is a compact open set in the constructible topology of $\vMod(A,n)$ and thus a constructible set. Since the complement of a constructible set is again constructible, also $\vMod(\mathcal{X},n)$ is constructible.
\end{proof}

Next, fix a constructible subcategory $\mathcal{X}$ of $\Mod A$. We would like to show that infinite type implies unbounded type for $\mathcal{X}$. To do so, we may assume that there are infinitely many non-isomorphic modules of dimension $n$ in $\mathcal{X}$ and deduce that $\mathcal{X}$ is of unbounded type. In particular, in this case $k$ must be an infinite field since otherwise there are only finitely many $A$-module structures on $k^n$ as $A$ is finitely generated. Throughout, this assumption is made. By the results of Andres Fernandez Herrero in the appendix we show the following.

\begin{prop}\label{ass} There is a prime ideal $p\subseteq C$ in $\vMod(\mathcal{X},n)$ and positive integer $l>0$ such that $C/p$ has Krull dimension 1 with the following property: There are $|k|$-many non-isomorphic indecomposable direct summands of the $A$-modules $M \otimes_C C/m$ for maximal ideals $m\supseteq p$ in $\vMod(\mathcal{X},n)$ with $\dim C/m \leq l$.
\end{prop}
\begin{proof} Geometrically, $p$ corresponds to a curve inside $\vMod(\mathcal{X},n)$. There is a group scheme $GL_n$ acting on $\vMod(A,n)$, see the appendix. Clearly, $\vMod(\mathcal{X}, n)$ is invariant under this action and $\vMod(\mathcal{X},n) \subseteq \vMod(A,n)$ is a constructible set by \Cref{constru}. By the assumption on $\mathcal{X}$ there is an infinite set of non-isomorphic $A$-modules of dimension $n$ over $k$ that correspond to $k$-points in $\vMod(\mathcal{X},n)$. Thus, all conditions are fulfilled to apply \Cref{thm: main result appendix} and the existence of the desired curve follows (take $p$ to be the generic point of the curve). 
\end{proof}

Recall that a prime ideal is \emph{regular} if the localisation at the prime ideal is a regular local ring. The set of regular elements in $\Spec(C/p)$ is open in the Zariski topology, so all but finitely many maximal ideals are regular (since $C/p$ has Krull dimension 1). In particular, there is a finite localisation of $C/p$, which we denote by $B$, such that every maximal ideal of $B$ is regular, so $B$ is a Dedekind domain. Further, all but finitely many maximal ideals $m \supseteq p$ of $C$ correspond to maximal ideals $\tilde{m}$ of $B$ such that $C/m \cong B/{\tilde{m}}$, so by \Cref{ass} there are infinitely many non-isomorphic indecomposable $A$-modules that appear as direct summands of $M\otimes_C B/\tilde{m}$ for maximal ideals $\tilde{m} \subseteq B$. The Ziegler spectrum of a Dedekind domain is well-understood.

\begin{exa}\label{ded} \cite[Example 9.5]{Ziegler} \rm The points in $\Ind B$ are
\begin{align*}
    B/m^i, \qquad \hat{B}_m, \qquad P_m, \qquad \Frac(B) 
\end{align*}
with $i>0$, $m\subseteq B$ maximal ideals, $\hat{B}_m$ the completion at $m$ and $P_m$ the injective hull of $B/m$. The following describes a basis of open neighbourhoods of each point.
    \begin{center}
        \begin{tabular}{ |c|c| }
        \hline basis of open neighbourhoods & point  \\ 
            \hhline{|=|=|} $\{B/m^i\}$ & $B/m^i$  \\
            \hline $\{\hat{B}_m, B/m^i \mid i > j\}$ with $j>0$ & $\hat{B}_m$\\ 
            \hline $\{P_m, B/m^i \mid i > j\}$ with $j>0$  & $P_m$\\
            \hline $\Ind B \setminus \{B/m_1^{i_1}, \dots, B/m_s^{i_s}\}$ for  & $\Frac(B)$\\
         maximal ideals $m_j \subseteq B$ and $i_j>0$ & \\
            \hline
        \end{tabular}
    \end{center}
\vphantom\\
\\
It follows that a set $\mathcal{U} \subseteq \Ind B$ is closed if it fulfills the following two conditions: If $\hat{B}_m \in \mathcal{U}$ or $P_m \in \mathcal{U}$ or $|\mathcal{U}| = \infty$, then $\Frac(B) \in \mathcal{U}$; if $\mathcal{U}$ contains  $B/{m}^i$ for infinitely many $i>0$, then $\hat{B}_m, P_m \in \mathcal{U}$.
\end{exa}

By this description of $\Ind B$, we show a dichotomy for constructible subcategories of $\Mod B$.

\begin{lem}\label{cased} Let $\mathcal{Y}$ be a constructible subcategory of $\Mod B$. Then $\mathcal{Y}$ is of infinite type if and only if $\Frac(B) \in \mathcal{Y}$. Moreover, in this case $B/m^i \in \fin \mathcal{Y}$ for all $i> 0$ and all but finitely many maximal ideals $m \subseteq B$. 
\end{lem}

\begin{proof} The modules $B/m^i$ with $m \subseteq B$ maximal and $i>0$ are precisely the finite dimensional indecomposable $B$-modules. For the closed set $\mathcal{U} = \mathcal{Y} \cap \Ind B$ the complement $\Ind B\setminus \mathcal{U}$ is compact by \Cref{int} (1). Thus, $\Ind B \setminus \mathcal{U}$ equals a finite union of the open neighbourhoods in \Cref{ded}. Now the claim follows by inspection.
\end{proof}

We investigate the consequences for the constructible subcategory $\mathcal{X} \subseteq \Mod A$. For a maximal ideal $m\subseteq B$ we call $T_m = \{M \otimes_C B/m^i \mid i>0\}$ the \emph{tube} at $m$. Note that $T_m$ is a collection of finite dimensional $A$-modules that are not necessarily indecomposable.

\begin{prop}\label{contains} The constructible subcategory $\mathcal{X} \subseteq \Mod A$ contains the tube $T_m$ for all but finitely many maximal ideals $m\subseteq B$.
\end{prop}
\begin{proof} Let $\mathcal{X} = \mathcal{X}_F$ for $F\in C(A)$ and $f\colon \Mod B \rightarrow \Mod A, X \mapsto M \otimes_C X$. Since $M$ is free of rank $n$ over $C$, it follows that $f$ commutes with filtered colimits and products. Thus, $\mathcal{Y} = \mathcal{X}_{F\circ f}$ is a constructible subcategory of $\Mod B$, which consists of all $X\in \Mod B$ with $f(X) \in \mathcal{X}$. By the choice ot $p$
\begin{align*}
    f(\Frac(B)) = M\otimes_C \Frac(B) \cong M \otimes_C \Frac(C/p) \in \mathcal{X},
\end{align*}
see \Cref{ass}. It follows that $\Frac(B) \in \mathcal{Y}$, so $B/m^i \in \mathcal{Y}$ for all $i>0$ and all but finitely many maximal ideals $m\subseteq B$ by \Cref{cased}. By the definition of $T_m$ it follows that $T_m \subseteq \mathcal{X}$ for all but finitely many maximal ideals $m\subseteq B$.
\end{proof}

\begin{rem} \label{gen}\rm The $A$-module $X =  M \otimes_C \Frac(B)$ is infinite dimensional over $k$ but has finite endolength, since it is of dimension $n$ over $\Frac(B) \subseteq \text{End}_A(X)$. It follows that $X$ decomposes into a (possibly infinite) direct sum of indecomposable modules of finite endolength that involves only finitely many isomorphism classes by \cite[Proposition 4.5]{Crawley-Boevey}. An infinite dimensional indecomposable module of finite endolength is called \emph{generic}.
\end{rem}

It is shown that the infinite dimensional $A$-module $M \otimes_C \Frac(B)$ controls the tubes $T_m$. This requires the following result.

\begin{lem}\label{disc} Let $A$ be a finitely generated $k$-algebra. For a finite dimensional indecomposable module $X \in \Mod A$ the set $\{X\}$ is closed and open in $\Ind A$.
\end{lem}

\begin{proof} By \Cref{embed} there exists a finite dimensional $k$-algebra $\tilde{A}$ such that $\Mod A$ is equivalent to a definable subcategory of $\Mod \tilde{A}$. The equivalence induces a homeomorphism between $\Ind A$ and a closed set in $\Ind \tilde{A}$, see Section 1.2, that maps finite dimensional modules to finite dimensional modules. Now the claim follows by Section 1.5 (4).
\end{proof}

\begin{lem}\label{many} Let $Y$ be an indecomposable finite dimensional $A$-module.
\begin{itemize} 
    \item[(1)] The module $M\otimes_C \Frac(B)$ has a generic direct summand.
    \item[(2)] For every maximal ideal $m\subseteq B$ there are infinitely many non-isomorphic indecomposable modules that appear as direct summands of modules in $T_m$. 
    \item[(3)] Let $m_j, j\in J$ be infinitely many maximal ideals of $B$ and $X_j \in T_{m_j}$ such that $Y$ is isomorphic to a direct summand of $X_j$ for all $j\in J$. Then $Y$ is a direct summand of $M\otimes_C \Frac(B)$.
\end{itemize} 
\end{lem}
\begin{proof} (1) Let $f\colon \Mod B\to \Mod A, X\mapsto M\otimes_C X$. Because $M$ is free of rank $n$ over $C$, it follows that $f$ commutes with filtered colimits and products. Let $U = \{B/m \mid m \subseteq B \text{ maximal ideal}\}$ and $V$ the set of all indecomposable direct summands of modules in $f(U)$ up to isomorphism. By \Cref{ass} the set $V$ is infinite. By Theorem 1.3 the closure $\overline{V}$ consists of all indecomposable direct summands of modules in $f(\overline{U})$. By \Cref{ded} the closure $\overline{U}$ equals $U \cup \{\Frac(B)\}$. Thus, if $M\otimes_C \Frac(B)$ had no generic direct summand, then $\overline{V}$ would consist of finite dimensional modules only. However, in this case $\overline{V}$ is an infinite discrete closed set in the compact space $\Ind A$, see Section 1.5 (1) and \Cref{disc}, which is a contradiction. It follows that $M\otimes_C \Frac(B)$ has a generic direct summand.

(2) For a maximal ideal $m \subseteq B$ let $U = \{B/m^i \mid i> 0\}$ and $V$ the set of indecomposable direct summands of modules in $f(U) = T_m$. By Theorem 1.3 the closure $\overline{V}$ consists of all indecomposable direct summands of modules in $f(\overline{U})$. Since $\Frac(B) \in \overline{U}$, see \Cref{ded}, it follows that $\overline{V}$ contains every indecomposable direct summand of $M\otimes_C \Frac(B)$. In particular, by (1) the set $\overline{V}$ contains a generic module. If $V$ is a finite set, then $\overline{V} = V$ by \Cref{disc}, which is a contradiction. Thus, $|V| = \infty$.

(3) The set $\Ind A \setminus \{Y\}$ is closed in $\Ind A$ by \Cref{disc}. Since $\{Y\}$ is compact, there is $F\in C(A)$ with $\Ind A \setminus \{Y\} = \mathcal{X}_F \cap \Ind A$ by Proposition 1.2. Consider again the functor $f\colon \Mod B\to \Mod A, X\mapsto M \otimes_C X$, which commutes with filtered colimits and products. Then $\mathcal{Y} = \mathcal{X}_{F\circ f}$ is a constructible subcategory of $\Mod B$, which consists of all $X \in \Mod B$ with $f(X) \in \mathcal{X}_F$. Now $f(X) \in \mathcal{X}_F$ if and only if $M\otimes_C X$ does not have a direct summand isomorphic to $Y$. Thus, if $Y$ is not a direct summand of $M\otimes_C \Frac(B)$, then $\Frac(B) \in \mathcal{Y}$ and so $B/m^i \in \mathcal{Y}$ for all $i>0$ and almost all maximal ideal $m\subseteq B$ by \Cref{cased}. However, by the assumptions $B/m_j^{i_j} \notin \mathcal{Y}$ for all $j\in J$ and some $i_j >0$. This is a contradiction and it follows that $Y$ is a direct summand of $M \otimes_C \Frac(B)$.
\end{proof}

By the above lemma we know that for every maximal ideal $m$ there are infinitely many non-isomorphic indecomposable modules appearing as direct summands of modules in the tube $T_m$. We show that there is no bound on the dimension of these indecomposable modules.

\begin{prop}\label{grow} Let $m \subseteq B$ be a maximal ideal and $V$ the set of indecomposable direct summands of modules in $T_m$. Then
\begin{align*}
    \sup\,\{ \dim X \mid X \in V\} = \infty.
\end{align*}
\end{prop}
\begin{proof} Since $B$ is a Dedekind domain, there is a sequence of monomorphisms
\begin{align*}
    B/m \longrightarrow B/m^2 \longrightarrow B/m^3 \longrightarrow \dots
\end{align*}
such that $\coker f_i \cong B/m^{i}$, where $f_i \colon B/m \longrightarrow B/m^{i+1}$ is the composition of the first $i$ morphisms in the sequence. Because $M$ is free over $C$, it follows that $M\otimes_C(-)$ is exact. Thus, there is a sequence of monomorphisms
\begin{align*}
    X_1 \longrightarrow X_2 \longrightarrow X_3 \longrightarrow \dots 
\end{align*}
in $T_m$ with $X_i = M \otimes_C B/m^i$ and $\coker g_i \cong X_i$ for $g_i = M \otimes_C f_i$. We say that an indecomposable direct summand $X$ of $X_i$ \emph{completely splits} if the composition
\begin{align*}
    X\longrightarrow X_i \longrightarrow X_{i+1} \longrightarrow \dots \longrightarrow X_j
\end{align*}
is a split monomorphism for all $j \geq i$. Otherwise, the composition is a radical morphism for some $j \geq i$. If $X$ completely splits, then the morphism $h\colon X \to \varinjlim X_j$ is a pure monomorphism by \cite[Corollary 2.1.3]{Prest}. Because $X$ is pure-injective, see \Cref{disc}, it follows that $h$ splits, which yields a decomposition
\begin{align*}
    X_1 \longrightarrow \dots \longrightarrow X_{i-1} \longrightarrow X\oplus \tilde{X}_{i} \xlongrightarrow{\begin{pmatrix}
        1_X & 0 \\
        0 & h_i 
    \end{pmatrix}}  X\oplus \tilde{X}_{i+1} \xlongrightarrow{\begin{pmatrix}
        1_X & 0 \\
        0 & h_{i+1} 
    \end{pmatrix}} \dots \tag{$\ast$} 
\end{align*}
with $X_j \cong X \oplus \tilde{X}_j$  and $h_i \colon \tilde{X}_j \to \tilde{X}_{j+1}$ for $j \geq i$.

We start by iteratively splitting off  indecomposable direct summands of $X_1$ as in $(\ast)$ if they completely split. If there is no such direct summand left, we continue with the indecomposable direct summands of $X_2$ (that did not split off yet), then $X_3$, then $X_4$, and so on, up to some $X_s$ with $s\geq 1$. We end up with a decomposition of the original sequence
\begin{align*}
    Y_1 \oplus Z_1 \xlongrightarrow{\begin{pmatrix}
        \alpha_1 & \beta_1 \\
        0 & \gamma_1
    \end{pmatrix}} &Y_2 \oplus Z_2 \xlongrightarrow{\begin{pmatrix}
        \alpha_2 & \beta_2 \\
        0 & \gamma_2
    \end{pmatrix}} \dots \xlongrightarrow{\begin{pmatrix}
        \alpha_{s-1} & \beta_{s-1} \\
        0 & \gamma_{s-1}
    \end{pmatrix}}  Y_s \oplus Z_s \\
    \xlongrightarrow{\begin{pmatrix}
        1_{Y_s} & 0 \\
        0 & \gamma_s
    \end{pmatrix}} &Y_s \oplus Z_{s+1} \xlongrightarrow{\begin{pmatrix}
        1_{Y_s} & 0 \\
        0 & \gamma_{s+1}
    \end{pmatrix}} Y_s \oplus Z_{s+2} \longrightarrow \dots,
\end{align*}
where $\alpha_i$ are split monomorphisms and for all $1 \leq i \leq s$ there is $j \geq i$ such that $\gamma_j \circ \dots \circ \gamma_i$ is a radical morphism. Thus, for $t \in \mathbb{N}$ we may choose $s$ big enough such that $\gamma_{s-1} \circ \dots \circ \gamma_1 $ is a composition of $t$-many radical morphisms. It follows that if $\sup \, \{\dim X \mid X \in V\} < \infty$, then $\gamma_{s-1} \circ \dots \circ \gamma_1 = 0$ for $s$ big enough by the Harada-Sai lemma \cite[Lemma 1.2]{HS}. We show that this yields a contradiction. 

For $l \geq s$, up to isomorphism, the map $g_l \colon X_1 \to X_{l+1}$ equals
\begin{align*}
    \begin{pmatrix}
        1_{Y_1} &0 \\
        0 & \gamma_{l}
    \end{pmatrix} \circ \dots \circ \begin{pmatrix}
        1_{Y_1} &0 \\
        0 & \gamma_{s} 
    \end{pmatrix} \circ \begin{pmatrix}
        \alpha_{s-1} & \beta_{s-1} \\
        0 & \gamma_{s-1} 
    \end{pmatrix} \circ \dots \circ \begin{pmatrix}
        \alpha_{1} & \beta_{1} \\
        0 & \gamma_{1} 
    \end{pmatrix} \\
    = \begin{pmatrix}
        \alpha_{s-1} \circ \dots \circ \alpha_1 & \delta \\
        0 & \gamma_{l} \circ \dots \circ \gamma_1 
    \end{pmatrix} = \begin{pmatrix}
        \alpha_{s-1} \circ \dots \circ \alpha_1 & \delta \\
        0 & 0
    \end{pmatrix} 
\end{align*}
for some $\delta \colon Z_1 \to Y_{s}$ that does not depend on $l$. Let $Y_s/X_1$ denote the cokernel of $(\alpha_{s-1}\circ \dots \circ \alpha_1\,\,\,\, \delta) \colon Y_1\oplus Z_1\to Y_s$. Then
\begin{align*}
    X_l \cong \coker g_l \cong  Y_s/X_1\oplus Z_{l+1} .
\end{align*}
Since $X_{l+1} \cong Y_s \oplus Z_{l+1}$, the indecomposable direct summands of $X_l$ and $X_{l+1}$ coincide, up to those of $Y_s$ and $Y_s/X_1$, for all $l \geq s$. It follows that there are only finitely many isomorphism classes of indecomposable modules that appear as direct summands of modules in $T_m$, which contradicts \Cref{many} (2).
\end{proof}

We have now shown everything that we want. The following summarizes the results.

\begin{thm}\label{main} Let $k$ be a field, $A$ a finitely generated $k$-algebra and $\mathcal{X} \subseteq \Mod A$ a constructible subcategory with infinitely many non-isomorphic modules of the same dimension in $\fin \mathcal{X}$. There are sets $t_i, i\in I$ of indecomposable modules in $\fin \mathcal{X}$ with the following properties.
\begin{itemize}
    \item[(1)] The cardinality of $I$ equals the cardinality of $k$.
    \item[(2)] For all $i \in I$ the dimension of modules in $t_i$ is unbounded.
    \item[(3)] There are finitely many indecomposable $A$-modules $Y_1, \dots, Y_s$ such that if an $A$-module $Y$ is isomorphic to a module in $t_i$ for infinitely many $i \in I$, then $Y \cong Y_j$ for some $1\leq j \leq s$.
\end{itemize}
\begin{proof}  Let $n \in \mathbb{N}$ such that there are infinitely many non-isomorphic modules of dimension $n$ in $\mathcal{X}$. This is precisely the assumption on $\mathcal{X}$ that we had throughout. Choose $M, C, B$ as before. By \Cref{contains} the tubes
\begin{align*}
    T_m = \{M \otimes_C B/m^i \mid i> 0\}
\end{align*}
are contained in $\mathcal{X}$ for all but finitely many maximal ideals $m \subseteq B$, which there are $|k|$-many of. Let $m_i, i\in I$ be the maximal ideals with $T_{m_i} \subseteq \mathcal{X}$ and let $t_i$ be the set of indecomposable direct summands of modules in $T_{m_i}$ for $i\in I$. Then for all $i\in I$ the dimension of modules in $t_i$ is unbounded by \Cref{grow}. It is left to show the claim (3).

Let $Y_1, \dots, Y_s$ be a complete list of finite dimensional indecomposable direct summands of $M \otimes_C \Frac(B)$ up to isomorphism, see \Cref{gen}. If an $A$-module $Y$ is isomorphism to a module in $t_i$ for infinitely many $i \in I$, then $Y\cong Y_j$ for some $1\leq j \leq s$ by \Cref{many} (3).
\end{proof}
\end{thm}

\begin{coro}\label{indstep} Let $k$ be an uncountable field, $A$ a finitely generated $k$-algebra and $\mathcal{X} \subseteq \Mod A$ a constructible subcategory such that there are infinitely many non-isomorphic modules of the same dimension in $\fin \mathcal{X}$. Then $\mathcal{X}$ is of strongly unbounded type.
\end{coro}
\begin{proof} Let $t_i, i \in I$ and $Y_1, \dots, Y_s$ be as in \Cref{main}. By (1) and (2) in \Cref{main} as well as the pigeonhole principle, there is a sequence of natural numbers $n_1 < n_2 < \dots$ such that for every $j>0$ there are $|k|$-many $i\in I$ that admit an indecomposable module $X\in t_i$ with $\dim X = n_j$. We may choose $n_1$ with $n_1 > \dim Y_t$ for all $1\leq t \leq s$. Then for all $j > 0$ the indecomposable modules $X \in t_i$ with $\dim X = n_j$ yield $|k|$-many isomorphism classes by \Cref{main} (3). It follows that $\mathcal{X}$ is of strongly unbounded type.
\end{proof}

It would be nice to get rid of the restriction on the cardinality of $k$ in the above result. We discuss a possible approach.

\begin{rem}\rm To extend \Cref{indstep} to countable fields we would have to control the dimension of the modules in $t_i, i\in I$ in the setting of \Cref{main}. Recall from the proof of \Cref{main} that $t_i$ equals the set of indecomposable direct summands of $M \otimes_C B/m_i^j$ with $j>0$, where $m_i \subseteq B, i\in I$ are almost all maximal ideals. Let $d_{i,j}$ be the maximum of the dimensions of the indecomposable direct summands of $M \otimes_C B/m_i^j$. If there are constants $c_1, c_2 > 0$ such that $c_1\cdot  j \leq d_{i,j} \leq c_2 \cdot j$
for infinitely many $i\in {I}$ and $j>0$, then it is possible to replicate the proof of \Cref{indstep} since the pigeonhole principle now applies as for uncountable fields. 

By \Cref{ass} there is $l > 0$ with $\dim B/m \leq l$ for infinitely many maximal ideals $m \subseteq B$. Thus, we may choose $c_2 = l \cdot n$, where $n$ is the rank of $M$ over $C$. It is left to show the existence of $c_1$. A potential method would be to control the decomposition of the sequence 
\begin{align*}
    0 \longrightarrow M \otimes_C B/m \longrightarrow M \otimes_C B/m^2 \longrightarrow \dots 
\end{align*}
in the proof of \Cref{grow} independent of the maximal ideal $m$ (at least for a suitable choice of infinitely many maximal ideals). Then, an effective use of the Harada-Sai lemma would yield $c_1$ as desired.
\end{rem}

Independent of the underlying field, we succeeded in proving a variant of the first Brauer-Thrall conjecture for constructible subcategories.

\begin{thm}\label{mthm} Let $k$ be a field, $A$ a finitely generated $k$-algebra and $\mathcal{X} \subseteq \Mod A$ a constructible subcategory. If $\mathcal{X}$ is of infinite type, then $\mathcal{X}$ is of unbounded type.
\end{thm}
\begin{proof} Suppose that $\mathcal{X}$ is of infinite but not of unbounded type. Then there are infinitely many non-isomorphic modules of the same dimension in $\fin \mathcal{X}$. It follows that $\mathcal{X}$ is of unbounded type by \Cref{main} (2), which is a contradiction. 
\end{proof}

\section{Brauer Thrall II}

We discuss to what extend the classical second Brauer-Thrall conjecture can be generalized to constructible subcategories. It states that infinite type implies strongly unbounded type for the module category of a finite dimensional algebra over an infinite field. For algebraically closed fields the conjecture is proven \cite{Bautista, Bongartz}. This also implies the case of perfect infinite fields; the general case remains open. I am thankful to Dave Benson for providing the following example, which already shows an obstacle for a possible generalization.

\begin{exa}\label{counter}\rm Consider the universal enveloping algebra of $\text{sl}_2(\mathbb{C})$
\begin{align*}
    A = \mathbb{C}\langle e,f,h\rangle/(he-eh = 2e, hf -fh = -2f, ef-fe = h).
\end{align*}
For every $n > 0$ there is, up to isomorphism, a unique indecomposable $A$-module of dimension $n$. Thus, $\Mod A$ is of infinite but not of strongly unbounded type. By \Cref{embed} it follows that there is a finite dimensional $\mathbb{C}$-algebra $\tilde{A}$ and a constructible subcategory $\mathcal{X}\simeq \Mod A$ of $\Mod \tilde{A}$ such that $\mathcal{X}$ if of infinite type but not of strongly unbounded type.  
\end{exa}

The above example shows that, in general, infinite type does not imply strongly unbounded type for constructible subcategories of the module category of a finite dimensional algebra over an infinite field. Thus, we must impose some assumptions on the algebra.

Let $k$ be a field and $A$ a finite dimensional $k$-algebra. For the definition of generic modules, recall \Cref{gen}. The algebra $A$ is \emph{tame} if for all $n\in \mathbb{N}$ there are only finitely many isomorphism classes of generic $A$-modules $X$ with $\lend(X) = n$. Moreover, $A$ is \emph{domestic} if there are only finitely many isomorphism classes of generic $A$-modules in total. These definitions coincide with the classical ones for algebraically closed fields by \cite[Theorem 4,4, Corollary 5.7]{Crawley-Boevey0}. 

\begin{rem}\label{big}\rm The following  argument follows Herzog \cite[Theorem 9.6]{Herzog}; the case $\mathcal{X} = \Mod A$ is originally due to Crawley-Boevey \cite[Theorem 9.6]{Crawley-Boevey}. Let $k$ be a field, $A$ a finite dimensional $k$-algebra and $\mathcal{X} \subseteq \Mod A$ a definable subcategory. If there are infinitely many non-isomorphic indecomposable $X \in \mathcal{X}$ with $\lend (X) \leq n$, then there is a generic $A$-module in $\mathcal{X}$. Indeed, if $\mathcal{X}$ had no generic module, then 
\begin{align*}
\{X \in \mathcal{X} \mid X \text{ is indecomposable and } \lend(X) \leq n\}    
\end{align*}
would be an infinite discrete closed set in the compact space $\Ind A$ by (1)-(4) in Section 1.5, which is a contradiction. Note that the statement also holds for a finitely generated $k$-algebra using the same argument together with \Cref{disc}.
\end{rem}

\begin{prop}\label{tame} Let $k$ be an algebraically closed field and $A$ a finite dimensional tame $k$-algebra. A constructible subcategory $\mathcal{X} \subseteq \Mod A$ is of strongly unbounded type if an only if there is a generic $A$-module in $\mathcal{X}$. In that case there exists an $A$-$B$-bimodule $M$ free over $B$ of rank $\lend(X)$ with $B = k[x,f^{-1}]$ such that
\begin{align*}
        \{M \otimes_B k[x]/(x-\lambda)^i \mid i >0\} \subseteq \mathcal{X}
\end{align*}
for all but finitely many $\lambda \in k$. Further, $M \otimes_B k[x]/(x-\lambda)^i$ are non-isomorphic indecomposable $A$-modules of dimension $i \cdot \lend(X)$.
\end{prop}

\begin{proof} The inequality $\lend (Y) \leq \dim Y$ holds for $Y\in \Mod A$. Hence, by \Cref{big} if $\mathcal{X}$ is of strongly unbounded type, then $\mathcal{X}$ has a generic $A$-module.

Assume that there is a generic $A$-module $X$ in $\mathcal{X}$. Because $A$ is tame, there is an $A$-$B$-bimodule $M$ free over $B$ of rank $\lend(X)$ with $B= k[x,f^{-1}]$ such that the functor $g\colon \Mod B \rightarrow \Mod A, Y \mapsto M \otimes_B Y$ preserves indecomposable modules and $M \otimes_B k(x) \cong X$ \cite[p. 242]{Crawley-Boevey0}. Let $\mathcal{X} = \mathcal{X}_F$ for $F\in C(A)$. Since $M$ is free over $B$ of finite rank, it follows that $g$ commutes with filtered colimits and products. Thus, $\mathcal{Y} = \mathcal{X}_{F\circ g}$ is a constructible subcategory of $\Mod B$, which consists of all $Y\in \Mod B$ with $g(Y) = M \otimes_B Y \in \mathcal{X}$. In particular, $M \otimes_B k(x) \cong X \in \mathcal
X$ implies $k(x) \in \mathcal{Y}$. Applying \Cref{cased} to $\mathcal{Y}$ for the Dedekind domain $B$, it follows that
\begin{align*}
    \{M\otimes_B k[x]/(x-\lambda)^i \mid i >0 \} \subseteq \mathcal{X}
\end{align*}
for all but finitely many $\lambda \in k$. Because $M\otimes_B k[x]/(x-\lambda)^i$ are indecomposable non-isomorphic $A$-modules of dimension $i \cdot \lend(X)$, the category $\mathcal{X}$ is of strongly unbounded type.
\end{proof}

Next, we guarantee the existence of a generic module under suitable assumptions. This requires the following definition. For a ring $R$ let $\mathcal{S}_{-1} = 0$ be the trivial Serre subcategory of $C(R)$. If $\alpha$ is an ordinal of the form $\alpha = \beta +1$, let $\mathcal{S}_\alpha$ be the Serre subcategory of all objects in $C(R)$ that become of finite length in the Serre quotient $C(R) /\mathcal{S}_{\beta}$. If $\lambda$ is a non-zero limit ordinal, then let $\mathcal{S}_{\lambda} = \bigcup_{\alpha < \lambda} \mathcal{S}_{\alpha}$. The \emph{Krull-Gabriel dimension} of $R$ equals the smallest ordinal $\alpha$ with $\mathcal{S}_{\alpha} = C(R)$ \cite{Geigle}. If no such $\alpha$ exists, then the Krull-Gabriel dimension of $R$ is not defined. The following result also holds for a finitely generated algebra with the same proof making use of \Cref{disc}.

\begin{lem}\label{geaan} Let $k$ be a field and $A$ a finite dimensional $k$-algebra such that the Krull-Gabriel dimension of $A$ is defined. For every infinite closed set $\mathcal{U} \subseteq \Ind A$ there is a generic $A$-module $X\in \mathcal{U}$.
\end{lem}

\begin{proof} If $\mathcal{U}$ contained only finite dimensional modules, then it would be an infinite discrete compact space by Section 1.5 (1) and (4), which is a contradiction. Thus, $\mathcal{U} \setminus \mod A$ is non-empty and by compactness we may choose a minimal non-empty closed set $\mathcal{V} \subseteq \mathcal{U} \setminus \mod A$. Now $\mathcal{V}$ corresponds to a maximal Serre subcategory $\mathcal{S}$ of $C(A)$ by Theorem 1.2. By maximality, the Serre quotient $C(A)/\mathcal{S}$ either contains a unique simple object or every object admits a dense chain of subobjects. In the second case, by \cite[Proposition 7.2, Lemma B.8]{Krause0} the Krull-Gabriel dimension of $A$ is not defined. Thus, we are in the first case and $\mathcal{V} = \{X\}$ for an indecomposable $A$-module $X$ of finite endolength by \cite[Proposition 6.23]{Krause0}. By construction of $\mathcal{V}$ the module $X$ is not finite dimensional, hence generic, and $X\in \mathcal{U}$.
\end{proof}

Let $k$ be a field and $A$ a finite dimensional $k$-algebra. A conjecture of Prest states that the Krull-Gabriel dimension of $A$ is defined if and only if $A$ is domestic \cite[Conjecture 9.1.15]{Prest}. This holds for various algebras, see \cite[Section 1]{Pastuszak} for a detailed discussion. What is known in general is that if $k$ is algebraically closed and the Krull-Gabriel dimension of $A$ is defined, then the algebra is tame, see for example \cite[Proposition 8.15]{Krause0}.

\begin{coro}\label{algcase} Let $k$ be an algebraically closed field, $A$ a finite dimensional $k$-algebra whose Krull-Gabriel dimension is defined and $\mathcal{X} \subseteq \Mod A$ a constructible subcategory. If $\mathcal{X}$ is of infinite type, then $\mathcal{X}$ is of strongly unbounded type.
\end{coro}
\begin{proof} If $\mathcal{X}$ is of infinite type, then the corresponding closed set $\mathcal{U} = \mathcal{X} \cap \Ind A$ is infinite. By \Cref{geaan} there is a generic $A$-module $X \in \mathcal{U}$. Since the Krull-Gabriel dimension of $A$ is defined and the field is algebraically closed, it follows that $A$ is tame and $\mathcal{X}$ is of strongly unbounded type by \Cref{tame}. 
\end{proof}

Similar to the classical second Brauer-Thrall conjecture, we show that the above result can be extended to perfect infinite fields. For a field extension $L \supseteq k$ and $k$-algebra $A$, extension of scalars yields an $L$-algebra $A \otimes_k L$. For an $A$-module $X$ the space $X \otimes_k L$ is naturally an $A \otimes_k L$-module. Moreover, for an $A\otimes_k L$-module $X$ restriction of scalars along $A \to A \otimes_k L$ yields an $A$-module $r(X)$.  

\begin{lem}\label{restra} Let $k$ be a field, $L \supseteq k$ a field extension, $A$ a $k$-algebra and $\mathcal{X}$ a constructible subcategory of $ \Mod A$. There is a constructible subcategory $\tilde{\mathcal{X}}$ of $\Mod A \otimes_k L$ such that 
\begin{itemize}
    \item[(1)] for all $X\in \mathcal{X}$ extension of scalars yields $X\otimes_k L \in \tilde{\mathcal{X}}$, and
    \item[(2)] for all $X \in \tilde{\mathcal{X}}$ restriction of scalars along $A \to A \otimes_k L$ yields $r(X) \in \mathcal{X}$.
\end{itemize} 
Further, if $\mathcal{X}$ is of infinite type, then $\tilde{ \mathcal{X}}$ is of infinite type.
\end{lem}
\begin{proof} Let $\mathcal{X} = \mathcal{X}_F$ for $F\in C(A)$. Restriction of scalaras along $A \to A\otimes_k L$ is a  functor  $r\colon \Mod A\otimes_k L \to \Mod A$ that commutes with filtered colimits and products. It follows that $\tilde{\mathcal{X}} = \mathcal{X}_{F\circ r}$ is a constructible subcategory of $\Mod A\otimes_kL$, which consists of all $X \in \Mod A \otimes_k L$ with $r(X) \in \mathcal{X}$. Thus, (2) holds. For $X \in \mathcal{X}$ 
\begin{align*}
   r(X \otimes_k L ) = \bigoplus_{b\in B} X \in \mathcal{X},
\end{align*}
where $B$ is a $k$-basis of $L$. Hence, $X\otimes_k L \in \tilde{\mathcal{X}}$ and (1) holds.

Assume that $\tilde{\mathcal{X}}$ is not of infinite type. Let $X_1, \dots, X_n$ be,  up to isomorphism, a complete list of indecomposable modules in $\fin \tilde{\mathcal{X}}$. For indecomposable $X\in \fin \mathcal{X}$
\begin{align*}
\bigoplus_{b\in B} X \cong  r( X \otimes_k L) \cong \bigoplus_{i=1}^n r(X_i)^{t_i},
\end{align*}
where $B$ is a $k$-basis of $L$ and $ X\otimes_k L\cong \bigoplus_{i=1}^n X_i^{t_i}$ with $t_i>0$. Note that the endolength of $r(X \otimes_k L)$ is finite, since the dimension over $L$ is finite. Comparing the decompositions, it follows from \cite[Proposition 4.5, Remark 4.5]{Crawley-Boevey} that $X$ is isomorphic to an indecomposable direct summand of $r(X_i)$ for some $1\leq i \leq n$, which there are finitely many of, up to isomorphism. It follows that $\mathcal{X}$ is not of infinite type.
\end{proof}

\begin{thm}\label{kgbtt} Let $k$ be a perfect infinite field, $\bar{k}$ its algebraic closure, $A$ a finite dimensional $k$-algebra such that the Krull-Gabriel dimension of $A\otimes_k \bar{k}$ is defined and $\mathcal{X} \subseteq \Mod A$ a constructible subcategory. If $\mathcal{X}$ is of infinite type, then $\mathcal{X}$ is of strongly unbounded type. 
\end{thm}
\begin{proof} Let $\bar{A} =  A\otimes_k \bar{k}$ and assume that $\mathcal{X}$ is of infinite type. By \Cref{restra} there is a constructible subcategory $\bar{\mathcal{X}} \subseteq \Mod \bar{A}$ such that restriction of scalars $r \colon \Mod \bar{A} \to \Mod A$ and $(-) \otimes_k \bar{k} \colon \Mod A \to \Mod \bar{A}$ restrict to $\mathcal{\bar{X}} \to \mathcal{X}$ and $\mathcal{X} \to \bar{\mathcal{X}}$ respectively. Further, $\bar{\mathcal{X}}$ is of infinite type. Thus, the corresponding closed set $\mathcal{U} = \bar{\mathcal{X}} \cap \Ind \bar{A}$ is infinite. Since the Krull-Gabriel dimension of $\bar{A}$ is defined, there is a generic $\bar{A}$-module $X\in \mathcal{U}$ by \Cref{geaan}. Recall that a finite dimensional algebra over an algebraically closed field is tame if the Krull-Gabriel dimension is defined. Thus, by \Cref{tame} there is an $\bar{A}$-$B$-bimodule $M$ free over $B$ of rank $n\in \mathbb{N}$ with $B = \bar{k}[x,f^{-1}]$ such that
\begin{align*}
    \{M \otimes_B \bar{k}[x]/(x-\lambda)^i \mid i > 0\} \subseteq \bar{\mathcal{X}}
\end{align*}
for all but finitely many $\lambda \in \bar{k}$. Further, $M \otimes_B \bar{k}[x]/(x-\lambda)^i$ are non-isomorphic indecomposable $\bar{A}$-modules of dimension $i \cdot n$.

There is $m \in \mathbb{N}$ such that $\bar{A}$ is generated by $m$-many elements as a $\bar{k}$-algebra. The $\bar{A}$-$B$-bimodule $M$ is given my $n\times n$ matrices $M_1,  \dots, M_m$ over $B = \bar{k}[x,f^{-1}]$. Let $F/k$ be the field extension generated by the finitely many elements in $\bar{k}$ that appear as the coefficients of the entries of $M_i$ for $1\leq i \leq m$. Then $M_1, \dots, M_m$ over $C = F[x, f^{-1}]$ define an $A \otimes_k F$-$C$-bimodule $N$ such that $N \otimes_F \bar{k} \cong M$. Consider restriction of scalars $r_1 \colon \Mod \bar{A} \to \Mod A \otimes_k F$ and $r_2 \colon \Mod A \otimes_k F \to \Mod A$, so $r = r_2 \circ r_1$. Then 
\begin{align*}
   \bigoplus_{b\in D} r_2(N\otimes_C F[x]/(x-\lambda)^i)= r (M \otimes_B \bar{k}[x]/(x-\lambda)^i) \in \mathcal{X},
\end{align*}
where $D$ is an $F$-basis of $\bar{k}$, for all $i>0$ and all but finitely many $\lambda \in F$, which shows $r_2(N \otimes_C F[x]/(x-\lambda)^i) \in \mathcal{X}$. Since $M \otimes_B \bar{k}[x]/(x-\lambda)^i$ are non-isomorphic indecomposable $\bar{A}$-modules of dimension $i \cdot n$, the isomorphism 
\begin{align*}
    (N \otimes_C F[x]/(x-\lambda)^i ) \otimes_F \bar{k} \cong M \otimes_B \bar{k}[x]/(x-\lambda)^i
\end{align*}
shows that $N \otimes_C F[x]/(x-\lambda)^i$ are non-isomorphic indecomposable $ A \otimes_k F$-modules of dimension $i \cdot n$. Because $k$ is perfect, the natural epimorphism
 \begin{align*} r_2 (N \otimes_C F[x]/(x- \lambda)^i) \otimes_k F\longrightarrow N \otimes_C F[x]/(x-\lambda)^i
 \end{align*}
 splits as $A\otimes_k F$-modules by \Cref{lemma: sep extension} (2). It follows that $N\otimes_C F[x]/(x-\lambda)^i$ is isomorphic to a direct summand of $L \otimes_k X_{\lambda,i}$, where $X_{\lambda,i}$ is an indecomposable direct summand of the $A$-module $r_2 (N \otimes_C F[x]/(x-\lambda)^i)$. In particular
 \begin{align*}
    i \cdot n\leq  \dim X_{\lambda,i} \leq [F:k] \cdot i \cdot n.  
 \end{align*}
 Now $X_{\lambda, i} \in \mathcal{X}$ for all but finitely many $\lambda \in F$ (thus infinitely many $\lambda \in F$, since $k \subseteq F$ is infinite) and all $i>0$.  Further, for fixed $\lambda$ we have $X_{\lambda,i} \cong X_{\mu,i}$ for at most finitely many $\mu$, since the modules $N \otimes_C F[x]/(x-\lambda)^i$ are non-isomorphic. In total, the existence of the $X_{\lambda,i}$ shows that $\mathcal{X}$ is of strongly unbounded type by the pigeonhole principle.
\end{proof}

\section{Exact Structures and Matrix Reductions}

Let $k$ be a field and $A$ a finite dimensional $k$-algebra. We provide examples of constructible subcategories of $\Mod A$ that arise from exact structures as well as matrix reductions. As a consequence, this yields a systematic method to translate a matrix reduction into a reduction of exact structures in the sense of \cite{BHLR}. This is done by connecting matrix reductions with the theory of purity, see Section 1. For exact structures this was already achieved in \cite{Sch}.

An exact structure $\mathcal{E}$ on $\mod A$ is a collection of short exact sequences in $\mod A$ subject to some closure properties, see \cite{Buehler} for more details. We say that $\mathcal{E}$ is \emph{finitely generated} if $\mathcal{E}$ equals the smallest exact structure on $\mod A$ containing a specific short exact sequence. The following shows that finitely generated exact structures on $\mod A$ not only yield examples of constructible subcategories of $\Mod A$, but give a different point of view on them.

\begin{rem}\label{exact}\rm Let $\inj A \subseteq \mod A$ be the full subcategory of injective $A$-modules. By {\cite[Theorem 3.1, Corollary 4.1]{Sch}} the following are one to one.
\begin{itemize}
    \item[(1)] Exact structures $\mathcal{E}$ on $\mod A$.
    \item[(2)] Definable subcategories $\mathcal{X}$ of $\Mod A$ containing $\inj A$.
\end{itemize}
Given an exact structure $\mathcal{E}$ on $\mod A$ the corresponding definable subcategory $\mathcal{X}_\mathcal{E}$ consists of all $X\in \Mod A$ that behave injectively with respect to $\mathcal{E}$, that is, for every short exact sequence $0 \to L \to M \to N \to 0$ in $\mathcal{E}$ and morphism $L \to X$ there is $M \to X$ making the diagram
\begin{equation*}
    \begin{tikzcd}
        & X \\
        0 \arrow[r] & L \arrow[u] \arrow[r] & M \arrow[lu] \arrow[r] & N  \arrow[r] & 0
    \end{tikzcd}
\end{equation*}
commute. Further, for the corresponding closed set $\mathcal{U}_\mathcal{E} = \mathcal{X}_\mathcal{E}\cap \Ind A$ in the Ziegler spectrum the complement $\Ind A \setminus \mathcal{U}_\mathcal{E}$ is compact if and only if $\mathcal{E}$ is finitely generated by \cite[Proposition 3.2]{Sch}. Thus, by \Cref{int} (1) the exact structure $\mathcal{E}$ is finitely generated if and only if $\mathcal{X}_\mathcal{E}$ is a constructible subcategory of $\Mod A$. 

We conclude that finitely generated exact structures $\mathcal{E}$ on $\mod A$ are one to one with constructible subcategories $\mathcal{X}$ of $\Mod A$ containing $\inj A$. Note that the condition $\inj A \subseteq \mathcal{X}$ is not a proper restriction, since we may freely add/remove finitely many finite dimensional  indecomposable modules to/from $\mathcal{X}$ and remain constructible. This follows from the characterization in \Cref{int} (1) and the fact that finite dimensional indecomposable modules are closed and open points inside $\Ind A$, see Section 1.5 (4).
\end{rem}

Let us now consider matrix reductions. There are various frameworks in which they can be performed. We consider a very general one in terms of lift categories, which was developed by Crawley-Boevey in \cite{CB}. One of its main features is that it works over arbitrary fields. In the case of algebraically closed fields it generalizes the framework of normal free triangular linear bocses, see \cite[p. 2]{CB}. A connection between generic modules and representation type, which was originally proven via lift categories \cite[Theorem 9.6]{Crawley-Boevey},  can partially be proven using the theory of purity \cite[Theorem 9.6]{Herzog}, see also \Cref{big}. This already hints at a connection between matrix reductions and purity.

Often, the first step in performing matrix reductions is replacing $\Mod A$ with a full subcategory of the category of morphisms  between projective $A$-modules, which we denote by $\textnormal{Proj}^2A$. The objects in $\textnormal{Proj}^2A$ are morphisms $P \to Q$ for projective $A$-modules $P, Q$ and the morphisms in $\textnormal{Proj}^2 A$ are given by suitable commutative squares
\begin{equation*}
    \begin{tikzcd}
        P \arrow[r]\arrow[d] & Q \arrow[d] \\
        \tilde{P} \arrow[r] & \tilde{Q}. 
    \end{tikzcd}
\end{equation*}
The full subcategory in question is given by all $\varphi\colon P \to Q$ in $\textnormal{Proj}^2 A$ fulfilling $\im \varphi \subseteq J_A Q$, which we denote by $\mathcal{P}_1 (A)$, where $J_A$ is the Jacobson radical of $A$. We also consider the full subcategory of all $\varphi\colon P \to Q$ in $\mathcal{P}_1(A)$ with $\ker \varphi \subseteq J_A P$, which we denote by $\mathcal{P}_2 (A)$.

\begin{rem}\label{defs}\rm It is known that for a finite dimensional algebra $A$ the category of projective $A$-modules $\Proj A$ is finitely accessible, see Section 1.4, and the full subcategory of finitely presented objects coincides with the category $\proj A$ of finite dimensional projective $A$-modules. It follows that also $\textnormal{Proj}^2 A$ is finitely accessible and the full subcategory of finitely presented objects coincides with the category $\textnormal{proj}^2A$ of morphisms between finite dimensional projective $A$-modules, see for example \cite[Remark 11.1.3]{Krause4}. The functors
\begin{align*}
    F\colon \textnormal{Proj}^2 A \longrightarrow \Ab,\quad &(\varphi\colon P \to Q ) \mapsto (J_AQ + \im \varphi)/J_A Q, \\
    G\colon \textnormal{Proj}^2 A \longrightarrow \Ab, \quad &(\varphi \colon P \to Q ) \mapsto (J_A P +\ker \varphi)/ J_A P
\end{align*}
commute with filtered colimits and products, because so does multiplication with $J_A$ as well as taking images and kernels. It follows that $F, G\in C(\textnormal{Proj}^2 A)$ and $\mathcal{X}_F$, respectively $\mathcal{X}_G$, is a definable subcategory of $\textnormal{Proj}^2 A$, which consist of all $\varphi\colon P \to Q$ with $\im \varphi \subseteq J_A Q$, respectively $\ker \varphi \subseteq J_A P$. Hence, $\mathcal{P}_1(A) = \mathcal{X}_{F}$ and $\mathcal{P}_2(A) = \mathcal{X}_F \cap \mathcal{X}_G$ are definable subcategories of $\textnormal{Proj}^2 A $.
\end{rem}

By the above remark $\Ind \mathcal{P}_2(A) \subseteq \Ind \mathcal{P}_1(A) \subseteq \Ind \textnormal{Proj}^2 A$ are closed subsets, see Section 1.2. We investigate them next.
 
\begin{lem}\label{discrm} Let $P \to Q$ in $\textnormal{proj}^2 A$ be indecomposable. Then $\{P \to Q\}$ is a closed and open set in $\Ind \textnormal{Proj}^2 A$. 
\end{lem}

\begin{proof} The category $\textnormal{Mod}^2 A$ of morphisms between $A$-modules is equivalent to the module category $\Mod \tilde{A}$ with 
\begin{align*}
    \tilde{A} = \begin{pmatrix}
        A & A \\
        0 & A
    \end{pmatrix}
\end{align*}
and the equivalence sends a morphism between finite dimensional $A$-modules to a finite dimensional $\tilde{A}$-module. We show that $\textnormal{Proj}^2 A$ is a definable subcategory of $\textnormal{Mod}^2 A$. In that case $\Ind \textnormal{Proj}^2 A$ identifies with a closed subset of $\Ind \textnormal{Mod}^2 A \cong \Ind \tilde{A}$ and the claim follows by Section 1.5 (4). 

The functor $f\colon \textnormal{Mod}^2 A \to \Mod A, (X\to Y) \mapsto X \oplus Y$ commutes with filtered colimits and products, since they are computed locally in $ \textnormal{Mod}^2 A$. Since $\Proj A$ is a definable subcategory of $\Mod A$, also the preimage $\textnormal{Proj}^2 A = f^{-1}(\Proj A)$ is a definable subcategory of $\textnormal{Mod}^2 A$. This follows from the closure properties of definable subcategories and the fact that $f$ preserves filtered colimits, products and pure-exact sequences, see Section 1.2 and Theorem 1.3.
\end{proof}

The following shows that the Ziegler spectra of $\mathcal{P}_1(A), \mathcal{P}_2(A),\textnormal{Proj}^2 A$ and $\Mod A$ are essentially the same, up to discrete finite subsets.

\begin{prop}\label{similar}  Let $P_1, \dots, P_n$ be all indecomposable projective $A$-modules up to isomorphism.
\begin{itemize}
    \item[(1)] There are disjoint unions 
\begin{align*}
    \Ind \mathcal{P}_2(A)\, \dot{\cup }\, \{P_1 \xrightarrow{} 0, \dots, P_n\xrightarrow{} 0\} &= \Ind \mathcal{P}_1(A),\\
    \Ind \mathcal{P}_1(A) \,\dot{\cup }\,\{P_1 \xrightarrow{1} P_1, \dots, P_n \xrightarrow{1} P_n\} &= \Ind \textnormal{Proj}^2 A.
\end{align*}
\item[(2)] The functor $f\colon \mathcal{P}_2(A) \to \Mod A, (\varphi\colon P \to Q) \mapsto \coker \varphi$ is full, dense and reflects isomorphisms. Further, $f$ commutes with filtered colimits and products, and induces a homeomorphism $\Ind \mathcal{P}_2(A) \to \Ind A$.
\end{itemize}
\end{prop}
\begin{proof} (1) Every morphism in $\textnormal{Proj}^2 A$ is isomorphic to a direct sum
\begin{align*}
     ({\varphi} \colon {P} \to {Q}) \oplus (\tilde{P} \xrightarrow{0} 0) \oplus (\tilde{Q} \xrightarrow{1} \tilde{Q})
\end{align*}
with $\im {\varphi} \subseteq J_A {Q}$ and $\ker {\varphi} \subseteq J_A {P}$, so ${\varphi}\in \mathcal{P}_2(A)$. Further, the zero morphism $(\tilde{P} \to 0) \in \mathcal{P}_1(A)\setminus (\mathcal{P}_2(A) \setminus 0 )$ is indecomposable if and only if $\tilde{P}$ is indecomposable, and the identity morphism $(\tilde{Q} \to \tilde{Q}) \in \textnormal{Proj}^2 A \setminus (\mathcal{P}_1(A) \setminus 0)$ is indecomposable if and only if $\tilde{Q}$ is indecomposable. This shows the claim with \Cref{discrm}.

(2) The first assertion is well-known and may be deduced by considering the first two terms of minimal projective resolutions, compare \cite[Theorem 1.7]{CB}. The functor $f$ commutes with filtered colimits and products, since so does taking cokernels. Thus, by Theorem 1.3 (1) $f$ preserves pure-injectivity. Because $f$ is also full, it follows that for every $\varphi \in \Ind \mathcal{P}_2(A)$ either $f(\varphi) = 0$ or $f(\varphi) \in \Ind A$. It is not hard to see that $f(\varphi) = 0$ implies $\varphi =0$. Hence, $f$ induces a map $\Ind \mathcal{P}_2(A)\to \Ind A$, which is bijective since $f$ is dense and reflects isomorphisms. By Theorem 1.3 (2) this map is a homeomorphism. 
\end{proof}

\begin{coro}\label{constrmat} There is a one to one correspondence 
\begin{equation*}
    \begin{Bmatrix}
        \text{constructible}\\
        \text{subcategories of }\mathcal{P}_1 (A) 
    \end{Bmatrix}  \longleftrightarrow 
    \begin{Bmatrix}
        \text{constructible}\\
        \text{subcategories of }\Mod A 
    \end{Bmatrix} \times \begin{Bmatrix}
        \text{subsets of}\\
        \{P_1, \dots, P_n\}
    \end{Bmatrix},
\end{equation*}
where $P_1, \dots, P_n$ are all projective $A$-modules up to isomorphism. 
\end{coro}
\begin{proof} We use the characterization of constructible subcategories in \Cref{int} (1). A constructible subcategory of $\mathcal{P}_1(A)$ corresponds to a closed set $\mathcal{V} \subseteq \Ind \mathcal{P}_1(A)$ such that $\Ind \mathcal{P}_1(A)  \setminus \mathcal{V}$ is compact. The set
\begin{align*}
\mathcal{U} = \{P_1 \to 0, \dots, P_n \to 0\} \subseteq \Ind \mathcal{P}_1(A)
\end{align*}
consists of finitely many closed and open singletons, see \Cref{discrm}. Such a closed set $\mathcal{V}$ corresponds to a pair $(\mathcal{V}_1, \mathcal{V}_2)$ via $\mathcal{V}_1 = \mathcal{V} \setminus \mathcal{U}$ and $\mathcal{V}_2 = \mathcal{V} \cap \mathcal{U}$ with the following properties. The set $\mathcal{V}_1 \subseteq \Ind \mathcal{P}_1 (A) \setminus \mathcal{U}$ is closed, $(\Ind \mathcal{P}_1(A) \setminus \mathcal{U}) \setminus \mathcal{V}_1$ is compact and $\mathcal{V}_2 \subseteq \mathcal{U}$ is arbitrary. Now the claim follows from the homeomorphism $\Ind \mathcal{P}_1(A) \setminus \mathcal{U} = \Ind \mathcal{P}_2(A)\cong \Ind A$, see \Cref{similar}.
\end{proof}

Recall that the first step in performing matrix reductions is to replace $\Mod A$ with $\mathcal{P}_1 (A)$. We show that the matrix reductions in the framework of lift categories give rise to constructible subcategories of $\mathcal{P}_1(A)$, which correspond to constructible subcategories of $\Mod A$ by \Cref{constrmat}, up to some finite information. Now the constructible subcategories of $\Mod A$ correspond to finitely generated exact structures on $\mod A$, again up to finite information, see \Cref{exact}. In this way, matrix reductions can be translated to exact structures. To achieve this, we first introduce the framework of lift categories.

We restrict to the finite dimensional setting over $k$. A \emph{lift pair} $(R, \xi)$ consists of a finite dimensional $k$-algebra $R$ and an exact sequence
\begin{align*}
    \xi \colon 0 \longrightarrow M \longrightarrow E \xlongrightarrow{\pi} R \longrightarrow 0 
\end{align*}
of finite dimensional $R$-$R$-bimodules. The corresponding \emph{lift category} $\xi (R)$ has as objects the pairs $(P,e)$, where $P$ is a projective $R$-module and $e$ is a section of the $R$-module map $\pi \otimes 1 \colon E \otimes_R P \rightarrow P$. Since $P$ is projective, such an $e$ always exists. A morphism from $(P,e)$ to $(P',e')$ is an $R$-module map $\theta \colon P \rightarrow P'$ making the diagram 
\begin{equation*}
    \begin{tikzcd}
    P \arrow[r, "e"] \arrow[d, "\theta", swap] & E \otimes_R P \arrow[d, "1 \otimes \theta"] \\
    P' \arrow[r, "e' "] & E \otimes_R P'
    \end{tikzcd}
\end{equation*}
commute.

\begin{exa}\label{start} \rm Let $R =   \left(\begin{smallmatrix}A&0\\0&A\end{smallmatrix}\right)$ and consider the short exact sequence
  \begin{align*}
      \xi \colon 0 \longrightarrow \begin{pmatrix}
 0 & J_A\\
 0 & 0
  \end{pmatrix} \longrightarrow \begin{pmatrix}
 A & J_A\\
 0 & A
  \end{pmatrix} \longrightarrow \begin{pmatrix}
 A & 0\\
 0 & A
  \end{pmatrix} \longrightarrow 0
  \end{align*}
of $R$-$R$-bimodules. An object in $\xi(R)$ is given by a projective $R$-module $P$, which corresponds to a pair $(P_1, P_2)$ of projective $A$-modules, and the section $e$, which is given by a morphism $\varphi \colon P_1 \to P_2$ with $\im \varphi \subseteq J_A P_2$. In that way $\xi(R)$ is equivalent to $\mathcal{P}_1(A)$, compare also \cite[Example 1.7, Theorem 1.7]{CB}.
\end{exa}

We show that performing a matrix reduction for a definable lift category yields a constructible subcategory, which is again a definable lift category. Since the starting lift category is definable, see \Cref{start} and \Cref{defs}, we remain definable when performing matrix reductions. We dot not know whether every lift category is definable.

\begin{lem}\label{filtprod} Let $(R, \xi)$ be a lift pair. Then $\xi(R)$ has filtered colimits and products.
\end{lem}
\begin{proof} Let $E$ be the middle term of $\xi$. Because $R$ and $E$ are finite dimensional over $k$, it follows that $\Proj R$ is closed under filtered colimits and products, and  $E\otimes_{R} (-)$ commutes with filtered colimits and products.

We show that $\xi (R)$ has products. Let $(P_i, e_i)_{i\in I}$ be a collection of objects in $\xi (R)$ and consider the product of $R$-modules $P = \prod P_i$. Then $P$ is projective. Let $e$ be the map
\begin{align*}
    P = \prod P_i \xrightarrow{(e_i)_{i\in I}} \prod E\otimes_R P_i \cong E\otimes_R \prod P_i = E\otimes_R P.
\end{align*}
It is straightforward to check that $(P,e)$ is a product of the $(P_i, e_i)$.

To show that $\xi (R)$ has filtered colimits, consider a directed system $(P_i, e_i)_{i\in I}$ of objects in $\xi(R)$. This yields a directed system $(P_i)_{i\in I}$ of projective $R$-modules. The filtered colimit of $R$-modules $P = \varinjlim P_i$ is again projective. Let $e$ be the map
\begin{align*}
        P = \varinjlim P_i \xrightarrow{\varinjlim e_i} \varinjlim E\otimes_R P_i \cong E\otimes_R \varinjlim P_i = E\otimes_R P.
\end{align*}
Again, it is easy to check that $(P,e)$ is a filtered colimit of the $(P_i, e_i)$.
\end{proof}

Fix a lift pair $(R, \xi)$ with $\xi \colon 0 \to M \to E \xrightarrow{\pi} R \to 0$. We give a brief summary of the matrix reductions for $\xi(R)$. For more details, see \cite[Section 3.1, Section 4.1]{CB}.

Pick a maximal sub-bimodule $N \subseteq M$ and let $J$ be the Jacobson radical of $R$. Consider the exact sequences 
\begin{align*}
\xi_N &\colon 0 \longrightarrow \bar{M} \longrightarrow \bar{E} \longrightarrow R \longrightarrow 0,\\
\xi_{NJ} &\colon  0 \longrightarrow  \bar{M}/(\bar{M}\cap (\bar{E}J+J\bar{E})) \longrightarrow \bar{E}/(\bar{E}J+J\bar{E}) \longrightarrow R/J \longrightarrow 0
\end{align*}
with $\bar{M} = M/N$ and $\bar{E} = E/N$. There are functors $\xi(R) \xlongrightarrow{\sigma_N} \xi_N(R) \xlongrightarrow{\rho_J} \xi_{NJ}(R/J)$ defined as follows. For $X = (P,e) \in \xi(R)$ let $\sigma_N(X) = (P, \bar{e})$, where $\bar{e}$ equals the composition $P \xrightarrow{e} E \otimes_R P \to \bar{E} \otimes_R P$. For $Z = (P,f) \in \xi_N (R)$ let $\rho_J(Z) = (R/J\otimes_R P, \bar{f})$, where $\bar{f}$ equals the composition
\begin{align*}
    R/J \otimes_R P\xlongrightarrow{1 \otimes f} R/J \otimes_R \bar{E} \otimes_R P \xlongrightarrow{p} \bar{E}/(\bar{E}J+J\bar{E}) \otimes_R P
\end{align*}
with $p$ the natural projection.

The category $\xi_{NJ}(R/J)$ is equivalent to $\Mod B$, where $B$ is either a hereditary finite dimensional algebra, or the product of a semisimple artinian ring and a matrix ring over a skew tensor algebra, see the proof of \cite[Lemma 9.1]{Crawley-Boevey}. In any case, $B$ is a finitely generated $k$-algebra. The \emph{reduction functor} $\textnormal{red} \colon \xi(R) \to \Mod B$ is the composition
\begin{align*}
    \xi(R) \xlongrightarrow{ \sigma_N} \xi_N(R) \xlongrightarrow{\rho_{NJ}} \xi_{NJ}(R/J) \simeq \Mod B.
\end{align*}
For a finite dimensional $B$-module $X \in \Mod B$ let $\textnormal{Add}\, X \subseteq \Mod B$ be the full subcategory of all direct summands of coproducts of $X$. Then $\textnormal{red}^{-1}(\textnormal{Add}\, X)$ is equivalent to $\xi_X(R_X)$ for a lift pair $(R_X, \xi_X)$, see \cite[Theorem 4.2, Theorem 3.1]{CB}.

\begin{prop}\label{reduction} The reduction functor $\textnormal{red} \colon \xi(R) \to \Mod B$ commutes with filtered colimits and products. If $\xi(R)$ is definable, then $\xi_X(R_X)$ is equivalent to a constructible subcategory of $\xi(R)$.
\end{prop}
\begin{proof} For a finite dimensional $R$-module $M$ the functor $M \otimes_R (-)$ commutes with filtered colimits and products. Further, filtered colimits and products in $\xi(R)$ are computed within $\Mod R$, see the proof of \Cref{filtprod}. Thus, by the definition of $\sigma_N$ and $\rho_{NI}$ both functors commute with filtered colimits and products. It follows that the reduction functor commutes with filtered colimits and products.

Let $\mathcal{U} = \{X_1, \dots, X_n\}$ be the set of all indecomposable direct summands of $X$, up to isomorphism. Then $\mathcal{U}$ is closed and open in $\Ind B$ by \Cref{disc}. In particular, the complement $\Ind B \setminus \mathcal{U}$ is compact by Section 1.5 (1). Thus, there is $F\in C(B)$ with $\mathcal{U} = \mathcal{X}_F \cap \Ind B$ by Proposition 1.2. Now $\mathcal{X}_F \subseteq \Mod B$ is the smallest definable subcategory of $\Mod B$ containing $X$. Hence, $\mathcal{X}_F = \textnormal{Add}\, X$ by \cite[Section 2.5 (7)]{Krause0}. If $\xi(R)$ is definable, then $F\circ \textnormal{red} \in C(\xi(R))$, since the reduction functor commutes with filtered colimits and products. Thus, $\mathcal{X}_{F\circ \textnormal{red}}$ is a constructible subcategory of $\xi(R)$, which consists of all $X\in \xi(R)$ with $\textnormal{red}(X) \in \mathcal{X}_F = \textnormal{Add}\, X$. This yields the desired result with the equivalence $\textnormal{red}^{-1}(\textnormal{Add}\, X) \simeq \xi_X (R_X)$.
\end{proof}

We call $(R_X, \xi_X)$ a \emph{reduction} of $(R, \xi)$ and denote it by $(R, \xi) \leadsto (R_X, \xi_X)$. Performing matrix reductions for $A$ is given by a chain of reductions
\begin{align*}
    (\xi_0, R_0) \leadsto (\xi_1, R_1) \leadsto \dots \leadsto (\xi_n, R_n)
\end{align*}
with $\xi_0(R_0) \simeq \mathcal{P}_1(A)$, see \Cref{start}.

\begin{coro}\label{matcon} For all $1\leq i \leq n$ the category $\xi_i(R_i)$ is equivalent to a constructible subcategory of $\mathcal{P}_1(A)$.
\end{coro}
\begin{proof} Since $\xi_0(R_0) \simeq \mathcal{P}_1(A)$ is definable, see \Cref{defs}, the claim follows by induction with \Cref{reduction} and \Cref{trans}.
\end{proof}

The concept of reductions of exact structures was introduced in \cite{BHLR} and it is described by considering chains 
\begin{align*}
    \mathcal{E}_0 \subseteq \mathcal{E}_1 \subseteq \dots \subseteq \mathcal{E}_n
\end{align*}
in the lattice of exact structures on $\mod A$. The authors introduce the notion of length and Gabriel-Roiter measure relative to an exact structure and investigate how these change when performing reductions of exact structures. Further, it was shown in \cite[Section 4]{BHLR} for a given example that matrix reductions translate to reductions of exact structures. The following theorem shows that this is always possible in the framework of lift categories.

\begin{thm}\label{translation} Let $k$ be a field and $A$ a finite dimensional $k$-algebra. Matrix reductions for $A$ in terms of lift categories translate to reductions of exact structures. Moreover, the exact structures involved are finitely generated.
\end{thm}
\begin{proof} The following are the objects involved for translating matrix reductions to reductions of exact structures, starting with a chain of reductions in the framework of lift categories:
\begin{align*}
       \text{matrix reductions:}\quad &(\xi_0, R_0) \leadsto (\xi_1, R_1) \leadsto \dots \leadsto (\xi_n, R_n)\\
       \text{constructible subcategories of $\mathcal{P}_1(A)$:} \quad & \mathcal{X}_0 \supseteq \mathcal{X}_1 \supseteq \dots \supseteq \mathcal{X}_n\\
       \text{constructible subcategories of $\Mod A$:}\quad & \mathcal{Y}_0 \supseteq \mathcal{Y}_1 \supseteq \dots \supseteq \mathcal{Y}_n\\
       \text{reductions of exact structures:} \quad & \mathcal{E}_0 \subseteq \mathcal{E}_1 \subseteq \dots \subseteq \mathcal{E}_n.
\end{align*}
By \Cref{matcon} the lift categories $\xi_i(R_i)$ are equivalent to constructible subcategories $\mathcal{X}_i$ of $\mathcal{P}_1(A)$. By \Cref{constrmat} these yield constructible subcategories $\mathcal{Y}_i$ of $\Mod A$, where $\mathcal{Y}_i$ is the smallest definable subcategory of $\Mod A$ containing $\coker \varphi$ with $(\varphi\colon P \to Q)  \in \mathcal{X}_i \cap \Ind \mathcal{P}_1(A)$. We may recover $\mathcal{X}_i$ from $\mathcal{Y}_i$, up to whether or not $(P_j \to 0) \in \mathcal{X}_i$, where $P_1, \dots, P_m$ are all indecomposable projective $A$-modules up to isomorphism. By \Cref{exact} the constructible subcategories $\mathcal{Y}_i \subseteq \Mod A$ yield finitely generated exact structures on $\mod A$ via
\begin{align*}
    \mathcal{E}_i = \{0 \to L \xrightarrow{f} M \to N \to 0 \text{ exact} \mid \coker \Hom_{A}(f, X) =0\text{ for all } X \in \mathcal{Y}_i\}
\end{align*}
Again, we may recover $\mathcal{Y}_i$ from $\mathcal{E}_i$, up to whether or not $I_j \in \mathcal{X}_i$, where $I_1, \dots, I_m$ are all indecomposable injective $A$-modules up to isomorphism.
\end{proof}

\appendix
\section{Curves in schemes of representations}

\begin{center}by Andres Fernandez Herrero
\end{center}

\medskip
\medskip

The main goal of this appendix is to prove \Cref{thm: main result appendix}. We work over a fixed infinite ground field $k$. We denote the cardinality of $k$ by $|k|$. For any two $k$-schemes $X,Y$, we denote by $X \times Y$ the fiber product over $k$. If $Y = \Spec(R)$ for some $k$-algebra $R$, we also write $X_R:= X \times \Spec(R)$.

\subsection{Points of smooth schemes and \'etale morphisms}
     For any point $x$ of a $k$-scheme $X$, we denote the residue field of $x$ by $k_x$.

\begin{lem} \label{lemma: finite extension etale morphisms}
    Let $\pi: X \to Y$ be an \'etale morphism between two finite type schemes over $k$. Then, there exists a positive integer $m>0$ such that for all points $y \in Y$ and every $x \in \pi^{-1}(y) \subset X$, the induced extension of residue fields $k_x \supset k_y$ is a separable extension of degree at most $m$.
\end{lem}
\begin{proof}
    By \cite[\href{https://stacks.math.columbia.edu/tag/02GL}{Tag 02GL}]{sp}, for any $y \in Y$ the preimage $\pi^{-1}(y)$ is isomorphic to $\Spec(A_y)$, where $A_y$ is a finite dimensional $k_y$-algebra of the form $A_y = \prod_{i} k_{x_i}$ for a finite set of separable field extensions $k_{x_i} \supset k_y$. It suffices to find a uniform upper bound for the degrees of the extensions $k_{x_i}$ that does not depend on $y$. Alternatively, it is enough to find a uniform upper bound on $\dim_{k_y}(A_y)$ that is independent of the choice of $y \in Y$. Note that $\dim_{k_y}(A_y)$ is the number of geometric connected components of the fiber $\pi^{-1}(y)$. By \cite[\href{https://stacks.math.columbia.edu/tag/055I}{Tag 055I}]{sp}, it follows that the function $y \mapsto \dim_{k_y}(A_y)$ is locally constructible. Since the topological space of $Y$ is Noetherian, there is a uniform upper bound for this function, as desired.
\end{proof}

\begin{lem} \label{lemma: density of points and smoothness}
    Let $X$ be a finite type scheme over $k$. If there exists a geometrically reduced $k$-scheme $Y$ and a scheme-theoretically dense morphism $Y \to X$, then there is a dense open subscheme $U \subset X$ such that $U$ is smooth.
\end{lem}
\begin{proof}
    In view of \cite[\href{https://stacks.math.columbia.edu/tag/056V}{Tag 056V}]{sp}, it suffices to show that $X$ is geometrically reduced over $k$. This, in turn, follows from \cite[Cor. 11.10.7]{ega4}.
\end{proof}

For any given positive integer $m$, we denote by $|X|_{\leq m}$ the set of closed points $x$ of $X$ such that the residue field $k_x \supset k$ is a separable extension of degree at most $m$.
\begin{lem} \label{lemma: cardinality of points in smooth schemes}
    Let $X$ be a smooth positive-dimensional scheme over $k$. Then, there is some positive integer $m>0$ such that the cardinality of $|X|_{\leq m}$ is at least $|k|$.
\end{lem}
\begin{proof}
    By \cite[\href{https://stacks.math.columbia.edu/tag/054L}{Tag 054L}]{sp} we may assume, after replacing $X$ with an open subscheme, that $X$ is of finite type and it admits an \'etale morphism $\pi: X \to \mathbb{A}^d_k$ with $d \geq 1$. Let $m$ be the integer obtained by applying \Cref{lemma: finite extension etale morphisms} to $\pi$. Let $V = \pi(X) \subset \mathbb{A}^d_k$ denote the open image of $\pi$. By our construction of $m$, it follows that every $k$-point $v \in V(k)$ has a preimage in $|X|_{\leq m}$. Hence it suffices to show that the cardinality of the set $V(k)$ is at least $|k|$. 
    
    Let $p_1: \mathbb{A}^d_k \to \mathbb{A}^1_k$ denote the projection to the first coordinate. This is a smooth morphism, and hence the image $W:= p_1(V) \subset \mathbb{A}^1_k$ is open. Hence, the complement $\mathbb{A}^1_k \setminus W$ consists of finitely many points. Since $\mathbb{A}^1_k(k) =k$, it follows that $W(k)$ has the same cardinality as $k$. For every $w \in W(k)$, the preimage $p_1^{-1}(w)$ is isomorphic to $\mathbb{A}^{d-1}_k$, and the intersection $V \cap p_1^{-1}(w) \subset \mathbb{A}^{d-1}_k$ is open and nonempty. Since the $k$-points of $\mathbb{A}^{d-1}_k$ are dense, it follows that the set of $k$-points of $V \cap p_1^{-1}(w) \subset \mathbb{A}^{d-1}_k$ must also be dense. Hence, there exists a $k$-point in the preimage $p_1^{-1}(w) \cap V$ for all $w \in W(k)$. This shows that the cardinality of $V(k)$ is at least $|k|$, as desired.
\end{proof}

\subsection{Orbits and stabilizers}Let $G$ be a finite type group scheme over $k$ acting on a finite type scheme $X$. The corresponding action morphism $a: G \times X \to X$ is defined at the level of (scheme-valued) points by $(g,x) \mapsto g \cdot x$. The stabilizer group scheme $\Stab$ is the finite type group scheme over $X$ defined by the Cartesian diagram
     \[
\begin{tikzcd}
 \Stab \ar[r]
  \ar[d] & G \times X \ar[d, "\varphi"] \\ X \ar[r, "\Delta"] &
  X \times X,
\end{tikzcd}
\]
where $\Delta: X \to X \times X$ is the diagonal and $\varphi: G \times X \to X \times X$ is given at the level of points by $\varphi(g, x) = (g\cdot x, x)$.

If $F \supset k$ is a field extension and $x: \Spec(F) \to X$ is an $F$-point of $X$, then we define the stabilizer group $G_x$ to be the fiber
     \[
\begin{tikzcd}
 G_x \ar[r]
  \ar[d] & \Stab \ar[d] \\ \Spec(F) \ar[r, "x"] &
  X
\end{tikzcd}
\]
of $\Stab \to X$ at $x$. Note that the stabilizer $G_x$ is a closed subgroup scheme of $G_F$. We define the orbit morphism $o_x: G_F \to X$ at the level of points by $g \mapsto g \cdot x$. Let $G_F/G_x$ denote the quotient, which is a finite type scheme over $F$. Then the orbit morphism $o_x: G_F \to X$ factors as $o_x: G_F \to G_F/G_x \xrightarrow{\overline{o}_x} X$, where $\overline{o}_x$ is a monomorphism.

We say that two closed points $x,y \in X$ are in the same $G$-orbit if one (or equivalently all) of the following equivalent conditions hold:
\begin{itemize}
    \item $y$ is in the image of the orbit morphism $o_x: G_{k_x} \to X$.
    \item $x$ is in the image of the orbit morphism $o_y$.
    \item The images of $o_x$ and $o_y$ intersect.
\end{itemize}

\begin{lem} \label{lemma: G-orbits vs G(k) orbits}
    Let $X$ be a finite type scheme over $k$ equipped with an action of a finite type group scheme $G$. Let $x,y \in X(k)$ be two $k$-points of $X$. Suppose that $G_x$ is smooth. Then, the following are equivalent:
    \begin{enumerate}
        \item[(i)] The points $x,y$ lie in the same $G$-orbit.
        \item[(ii)] There is some point $g \in G(\overline{k})$ defined over a separable closure $\overline{k} \supset k$ such that $g \cdot x = y$, where we view $x$ and $y$ as $\overline{k}$-points of $X$ by base-extension.
    \end{enumerate}
\end{lem}
\begin{proof}
    The implication (ii) $\Rightarrow$ (i) is clear. For the converse, assume that (i) holds. Consider the subfunctor $\text{Transp}_{x,y} \subset G$ whose $T$-valued points for any $k$-scheme $T$ consist of elements $g \in G(T)$ such that $g \cdot x_T =y_T$. By \cite[VIB, Ex. 6.2.4(e)]{sga3}, we have that $\text{Transp}_{x,y}$ is represented by a scheme of finite type over $k$. Note that (i) is equivalent to $\text{Transp}_{x,y}$ being nonempty. Let $F \supset k$ be a field extension such that there exists an $F$-point $g \in \text{Transp}_{x,y}(F) \subset G(F)$. It follows directly from the definition that the base-change $(\text{Transp}_{x,y})_F \subset G_F$ is the translate $g \cdot (G_x)_F \subset G_F$. In particular, since $(G_x)_F$ is smooth over $F$, it follows that $(\text{Transp}_{x,y})_F$ is smooth over $F$. It then follows by \cite[\href{https://stacks.math.columbia.edu/tag/02V4}{Tag 02V4}]{sp} that $\text{Transp}_{x,y}$ is smooth over $k$. In view of \cite[\href{https://stacks.math.columbia.edu/tag/056U}{Tag 056U}]{sp}, it follows that $\text{Transp}_{x,y}$ has a $\overline{k}$-point, which shows the existence of the required element $g \in G(\overline{k})$. 
\end{proof}

\begin{lem} \label{lemma: infinitely many orbits open}
    Let $X$ be a finite type scheme over $k$ equipped with an action of a finite type group scheme $G$. Suppose that the following hold:
    \begin{enumerate}
        \item[(A)] There is an integer $r \geq 0$ such that for all $x \in X(k)$ we have $\dim(G_x) = r$.
        \item[(B)] The disjoint uniont $\bigsqcup_{x \in X(k)} o_x: \bigsqcup_{x \in X(k)} G \to X$ of orbit morphisms is scheme-theoretically dense. 
        \item[(C)] $X$ contains infinitely many $k$-points lying in distinct $G$-orbits.
    \end{enumerate}
    Then, the following hold:
    \begin{enumerate}
        \item[(i)] $\dim(X) > \dim(G) - r$.
        \item[(ii)] For every nonempty open $G$-stable subscheme $U \subset X$ with $\dim(U) = \dim(X)$ we have that $U$ contains infinitely many $k$-points lying in distinct $G$-orbits.
    \end{enumerate}
\end{lem}
\begin{proof}
    Fix a set $\Sigma \subset X(k)$ of representatives of $G$-orbits in $X(k)$. By assumption (C), we have that $\Sigma$ is infinite.
    
    We first show (i). One of the finitely-many irreducible components of $X$, say $Y \subset X$, contains an infinite subset $\Xi \subset \Sigma$. If $x \in \Xi$, then the orbit morphism $o_x: G \to X$ factors as $o_x: G \to G/G_x \xrightarrow{\overline{o}_x} X$, where $\overline{o}_x$ is a monomorphism with source the finite type scheme $G/G_x$ of dimension $\dim(G)-r$. Suppose for the sake of contradiction that $\dim(Y) \leq \dim(G)-r$. Then for every $x \in \Xi$ the image of the quasifinite morphism $\overline{o}_x: G/G_x \to X$ must contain a nonempty open subset of $Y$ by dimension reasons. But then, by the irreducibility of $Y$, for any choice of points $x \neq y \in \Xi$, we get that the images of $\overline{o}_x$ and $\overline{o}_y$ intersect, contradicting the fact that $x$ and $y$ were not in the same $G$-orbit.
    
    Let us now prove (ii). Assume for the sake of contradiction that $U$ contains finitely many $G$-orbits of $k$-points. Let $x_1, x_2, \ldots x_m \in U(k)$ be representatives of these orbits, which we may assume to lie in $\Sigma$ without loss of generality. The closed complement $Z := X \setminus U$ contains all other $G$-orbits of $k$-points. It follows that the scheme-theoretic image of $\varphi:= \bigsqcup_{x \in \Sigma} o_x: \bigsqcup_{x \in \Sigma} G \to X$ is contained in the union of $Z$ and the scheme-theoretic image of the finitely many orbits $\bigsqcup_{i=1}^m o_{x_i}: \bigsqcup_{i=1}^m G \to X$. Since $\varphi$ is scheme-theoretically dense by assumption (B), we must have that the induced morphism $\bigsqcup_{i=1}^m o_{x_i}:\bigsqcup_{i=1}^m G \to U$ is scheme-theoretically dense. This is equivalent to the morphism $\bigsqcup_{i=1}^m \overline{o}_{x_i}: \bigsqcup_{i=1}^m G/G_{x_i} \to U$ being scheme-theoretically dense. Since for all $i$ we have $\dim(G/G_{x_i}) = \dim(G)-r < \dim(X) = \dim(U)$, this is impossible.
\end{proof}

\subsection{Curves passing through distinct orbits}

\begin{prop} \label{prop: main prop curves}
    Let $X$ be a scheme of finite type over $k$ equipped with an action of a finite type smooth group scheme $G$ such that the following hold:
    \begin{enumerate}
        \item[(A)] For every field extension $F \supset k$ and every point $x: \Spec(F) \to X$, the stabilizer $G_x$ is smooth over $\Spec(F)$.
        \item[(B)] There are infinitely many distinct $G$-orbits of $k$-points in $X(k)$.
    \end{enumerate}
    Then, there exists a closed curve $C \hookrightarrow X$, a positive integer $m>0$, and a set of closed points $\Sigma' \subset |C|_{\leq m}$ such that:
    \begin{enumerate}
        \item[(i)] $\Sigma'$ has cardinality $|k|$, and so in particular it is infinite.
        \item[(ii)] The points in $\Sigma'$ lie in distinct $G$-orbits.
    \end{enumerate}
\end{prop}
\begin{proof}
    Let $\Sigma$ denote a set of representatives of $G$-orbits in $X(k)$. By assumption, $\Sigma$ is infinite. By upper-semicontinuity of fiber dimension for group schemes \cite[VIB, Prop. 4.1]{sga3}, given an integer $r$ the subset $X_r \subset X$ of points whose stabilizer has dimension $r$ is locally closed inside $X$. Since $X$ is of finite type, there must be finitely many $r$ such that $X_r$ is nonempty. Hence, there exists a fixed integer $r$ such that $X_r$ contains an infinite subset of the $k$-points in $\Sigma$. If we equip $X_r$ with the reduced subscheme structure, then the action of the geometrically reduced group scheme $G$ preserves $X_r$. After replacing $X$ with $X_r$, we may assume without loss of generality that all stabilizers of points in $X$ have the same dimension. 
    
    Let $\varphi:= \bigsqcup_{x \in \Sigma} o_x: \bigsqcup_{x \in \Sigma} G \to X$ denote the disjoint union of orbit morphisms for $x \in \Sigma$. The scheme-theoretic image of $\varphi$ is a reduced closed subscheme $Y \subset X$ that is stable under the action of the geometrically reduced group scheme $G$. Replacing $X$ with $Y$, we may assume that $\varphi$ has scheme-theoretic image $X$. Let $X^{sm} \subset X$ be the maximal smooth open subscheme of $X$. Note that $X^{sm}$ is $G$-stable. Since the morphism $\varphi$ is scheme-theoretically dense, it follows by \Cref{lemma: density of points and smoothness} that $X^{sm}$ is open and dense. By \Cref{lemma: infinitely many orbits open}, $X^{sm}$ contains infinitely many distinct $G$-orbits of $k$-points. Hence, after replacing $X$ with $X^{sm}$, we may assume without loss of generality that $X$ is smooth. 
    
    Furthermore, if we denote by $U \subset X$ the maximal open subscheme $U \subset X$ such that the restriction $\Stab|_{U} \to U$ of the stabilizer group scheme is flat, then $U$ is dense by generic flatness \cite[\href{https://stacks.math.columbia.edu/tag/052B}{Tag 052B}]{sp}. Since $U$ is $G$-stable, after applying \Cref{lemma: infinitely many orbits open} we may replace $X$ with $U$ and assume that $\Stab \to X$ is flat. By our assumption (A), the fibers of $\Stab \to X$ are smooth, and hence we conclude that $\Stab \to X$ is smooth. 
    
    By \cite[\href{https://stacks.math.columbia.edu/tag/06QJ}{Tag 06QJ}]{sp}, the quotient stack $X/G$ is a gerbe over an algebraic space $M$, which is quasiseparated since $G$ is of finite type. Let us denote by $\pi: X \to X/G \to M$ the induced morphism. Since $G$ is smooth, so is the quotient morphism $X \to X/G$. On the other hand, since $\Stab \to X$ is smooth, it follows that the gerbe $X/G \to M$ is also smooth. Hence $\pi:X \to M$ is smooth. By \cite[\href{https://stacks.math.columbia.edu/tag/06NH}{Tag 06NH}]{sp}, there is an open affine dense subscheme $U \subset M$. The preimage $\pi^{-1}(U) \subset X$ is a $G$-stable open dense subscheme of $X$, and hence by using \Cref{lemma: infinitely many orbits open} once again, we may replace $X$ with $\pi^{-1}(U)$ and assume that the coarse quotient $M$ is an affine scheme of finite type. Since $X$ is smooth and the morphism $\pi: X \to M$ is flat, of finite presentation and surjective, it follows that $M$ is smooth. After passing to a connected component of $M$ of maximal dimension and applying \Cref{lemma: infinitely many orbits open} again, we may assume that $M$ is irreducible. 
    
    Note that two $k$-points $x,y \in X(k)$ are in the same $G$-orbit if and only if their images $\pi(x)$ and $\pi(y)$ are distinct points of $M$. In particular, the $k$-points in the subset $\pi(\Sigma) \subset M(k)$ are distinct, and we must have $\dim(M)>0$. Choose a $k$-point $m \in M(k)$ with maximal ideal $\mathfrak{m}_m \subset \cO_M$. By smoothness, the tangent space $\mathfrak{m}_m/\mathfrak{m}_m^2$ is a $k$-vector space of dimension $d:= \dim(M)>0$. Choose $(d-1)$-many functions $f_1,\ldots, f_{d-1} \in \cO_M$ that vanish at $m$ and whose corresponding images in $\mathfrak{m}_m/\mathfrak{m}_m^2$ span a $(d-1)$-dimensional vector subspace. Then, by the Jacobian criterion for smoothness, if we denote by $Z = V(f_1, \ldots, f_{d-1}) \subset M$ the vanishing locus of these functions, then there is an open neighborhood $U \subset M$ of $m$ such that the intersection $Z^{\circ}:= Z \cap U$ is a smooth irreducible curve. Let $\eta_{Z^{\circ}}$ denote the generic point of $Z^{\circ}$, and consider the fiber $X_{\eta_{Z^{\circ}}}$ of $\eta_{Z^{\circ}}$ under $\pi: X \to M$. Since $\pi$ is smooth, we have that $X_{\eta_{Z^{\circ}}}$ is smooth over $\eta_{Z^{\circ}}$. Hence, by \cite[\href{https://stacks.math.columbia.edu/tag/056U}{Tag 056U}]{sp}, it contains a closed point $\eta \in X_{\eta_{Z^{\circ}}}$ whose residue field is a finite separable extension of the residue field of $\eta_{Z^{\circ}}$. It follows that the residue field of $\eta$ has transcendence degree $1$ over $k$, and the closure $C:= \overline{\eta} \subset X$ is a curve that dominates the closure $\overline{Z^{\circ}} \subset M$. To conclude the proof, we show that there is an integer $m>0$ and a subset $\Sigma \subset |C|_{\leq m}$ satisfying properties (i) and (ii) in the statement of the proposition.
    
    Let $\psi: C \to \overline{Z^{\circ}}$ denote the morphism obtained by restricting $\pi$ to $C$. Since the morphism $\psi$ of integral curves induces a separable extension of fraction fields by construction, there is some nonempty open $V \subset \overline{Z^{\circ}}$ such that the restriction $\psi_V: \psi^{-1}(V) \to V$ is \'etale. After further shrinking $V$, we may assume that $V \subset Z^{\circ}$ and that $\psi_V$ is surjective. Since $V$ is positive-dimensional and smooth, it follows from \Cref{lemma: cardinality of points in smooth schemes} that there is some positive integer $\ell>0$ such that cardinality of $|V|_{\leq m}$ is at least $|k|$. Let $q>0$ denote the positive integer obtained by applying \Cref{lemma: finite extension etale morphisms} to the \'etale surjective morphism $\psi|_V: \psi^{-1}(V) \to V$. For any $v \in |V|_{\leq \ell}$, it follows from \Cref{lemma: finite extension etale morphisms} that there is a point $\widetilde{v} \in |C|_{\leq q \cdot \ell}$ contained in the preimage $(\pi_V)^{-1}(v)$. As we range over $v \in |V|_{\leq \ell}$, we obtain a subset $\Sigma' \subset |C|_{\leq \ell \cdot q}$ of cardinality $|k|$ consisting of points which lie in distinct $G$-orbits, as desired. 
\end{proof}

\subsection{Equivariant decomposition of constructible subsets}

\begin{lem} \label{lemma: equivariant constructible subsets}
    Let $X$ be a finite type scheme over $k$ equipped with the action of a connected smooth affine group $G$. Let $H \subset |X|$ be a constructible subset of the underlying topological space $|X|$ of $X$ which is preserved by the action of $G$. Then, there exist a finite set $X_1, X_2, \ldots, X_{\ell} \subset X$ of pairwise disjoint $G$-stable locally closed subschemes such that $H = \bigsqcup_{i=1}^{\ell} |X_i|$.
\end{lem}
\begin{proof}
    To any such constructible subset $H$ we can assign a pair $(m,h)$ of nonnegative integers, where $h$ is the maximal dimension of closures $\overline{\eta} \subset X$ of points $\eta \in H$, and $m$ is the number of points of $H$ whose closure has dimension $h$. We argue by induction on the set of pairs $(m,h)$ equipped with lexicographic order.
    
    Let $\eta$ be a point of $H$ such that $Z:= \overline{\eta} \subset X$ has maximal dimension $h$. Then $H \cap |Z|$ is a dense constructible subset of the variety $Z$, and hence it contains a unique maximal open subset $|Z^{\circ}| \subset H \cap |Z|$. We denote by $Z^{\circ} \hookrightarrow X$ the locally closed subvariety of $X$ with underlying topological space $|Z^{\circ}|$ equipped with the reduced subscheme structure. Since the difference $|H| \setminus |Z^{\circ}|$ is a constructible set whose associated pair $(m,h)$ strictly smaller in lexicographic order, by induction it is enough to show that $Z^{\circ}$ is preserved by the action of $G$.

    Consider the action morphism $a_Z: G \times Z \to X$ given by $(g,z)\mapsto g \cdot z$, and let $\widetilde{Z}$ denote the scheme theoretic image of $a_Z$. Since $G \times Z$ is irreducible, it follows that $\widetilde{Z}$ is irreducible. Let $\widetilde{\eta}$ denote the generic point of $\widetilde{Z}$. By construction, $\widetilde{Z} \supset Z$, and hence $\widetilde{\eta}$ is a generalization of $\eta$. Note that, since $Z^{\circ} \subset Z$ is dense, the (constructible) image of $a_{Z^{\circ}}: G \times Z^{\circ} \to \widetilde{Z}$ is dense. Since $H$ is preserved by the action of $G$ and $Z^{\circ} \subset H$, the image of $a_{Z^{\circ}}$ is contained in $H$. We conclude that the constructible set $H \cap |\widetilde{Z}|$ is dense in $\widetilde{Z}$, and hence we must have that $\widetilde{\eta} \in H$. By maximality of $\dim(Z)  = \dim(\overline{\eta})$, this forces $\eta = \widetilde{\eta}$ and $\widetilde{Z} = Z$. We conclude that $Z= \widetilde{Z}$ is preserved by the action of $G$. Consider $Z \setminus H$, which is a constructible set preserved by the action of $G$. By induction, we may write $Z \setminus H = \sqcup_{i=1}^{\ell} |X_i|$ where $X_i \hookrightarrow Z$ are locally closed subschemes preserved by $G$. Each closure $\overline{X}_i$ is also preserved by the action of $G$. Note that open subset $|Z^{\circ}| \subset Z$ is the complement $|Z| \setminus (\bigcup_{i=1}^{\ell}|\overline{X}_i|)$ and hence it is also preserved by the action of $G$, as desired.
\end{proof}

\subsection{Orbits and stabilizers of schemes of representations}
Let $A$ be a finitely generated algebra over $k$, and fix a positive integer $n$. We denote by $\vMod(A,n)$ the affine scheme of finite type that parameterizes actions of $A$ on the standard $k$-vector space $k^{\oplus n}$ of dimension $n$. The group scheme $G:= GL_n$ acts on $\vMod(A,n)$ by changing the basis of the standard vector space.
\begin{lem} \label{prop: smooth stabilizers for representation scheme}
    For all field extensions $F \supset k$ and all points $x: \Spec(F) \to \vMod(A,n)$, the stabilizer $G_x$ is smooth over $F$.
\end{lem}
\begin{proof}
    The point $x$ corresponds to an action of the $F$-algebra $A \otimes_k F$ on the $F$-vector space $F^{\oplus n}$. The stabilizer $G_x$ is a closed subgroup scheme of $(GL_n)_F$, which can be described at the level of scheme-valued points as follows. For every $F$-algebra $R$, the group of $R$-points $G_x(R)$ is the subgroup of $GL_n(R)$ consisting of automorphisms $g \in GL_n(R)$ of the $R$-module $R^{\oplus n}$ that are compatible with the induced action of $A\otimes_k R$ on $R^{\oplus n}$. In other words, $G_x(R)$ is the group of automorphisms of the $R$-representation $R^{\oplus n}$ of the $R$-algebra $A \otimes_k R$.
    
    Let $\End \subset M_{n\times n}(F) = F^{n^2}$ denote the $F$-vector subspace of endomorphisms of $F^{\oplus n}$ that are compatible with the action of $A \otimes_k F$. We may think of $\End$ as an affine space $\mathbb{A}^N_F$ over $F$, equipped with a closed immersion of affine schemes $\End \hookrightarrow \mathbb{A}^{n^2}_F$. Then, our description of the functor of points of $G_x$ above shows that there is a Cartesian diagram of $F$-schemes      
    \[
\begin{tikzcd}
 G_x \ar[r]
  \ar[d] & \End \ar[d] \\ (GL_n)_F \ar[r] &
  \mathbb{A}^{n^2}_F,
\end{tikzcd}
\]
where the bottom horizontal arrow is the standard open immersion of $(GL_n)_F$ into the affine space $\mathbb{A}^{n^2}_F$ of matrices. It follows that $G_x$ is an open subscheme of the affine space $\End \cong \mathbb{A}^N_F$, and hence it is smooth over $F$.
    
\end{proof}

\begin{lem} \label{lemma: characterization of GL_n orbits}
    Let $x,y \in |\vMod(A,n)|$ be two closed points with residue fields $k_x,k_y$ which are separable over $k$. Let $M_x$ (resp. $M_y$) denote the $n$-dimensional representation of $A\otimes_k k_x$ (resp. $A \otimes_k k_y$) corresponding to the point $x$ (resp. $y$). Fix the choice of a separable closure $\overline{k} \supset k$ containing both $k_x$ and $k_y$. Then, the following are equivalent:
    \begin{enumerate}
        \item[(i)] The points $x,y$ are the in the same $GL_n$-orbit.
        \item[(ii)] $M_x \otimes_{k_x} \overline{k}$ and $M_y \otimes_{k_y} \overline{k}$ are isomorphic as representations of $A \otimes_k \overline{k}$.
    \end{enumerate}
\end{lem}
\begin{proof}
    By definition, the points $x,y$ lie in the same orbit if and only if the corresponding $\overline{k}$-points $\widetilde{x}, \widetilde{y} \in \vMod(A,n)(\overline{k})$ obtained by base-extension lie in the same $GL_n$-orbit. We may equivalently think of $\widetilde{x}$ and $\widetilde{y}$ as $\overline{k}$-points of the base-change $\vMod(A,n)_{\overline{k}} \cong \vMod(A\otimes_k \overline{k}, n)$ equipped with its action of $(GL_n)_{\overline{k}}$. The points $\widetilde{x}$ and $\widetilde{y}$ correspond to the representations $M_x \otimes_{k_x} \overline{k}$ and $M_y \otimes_{k_y} \overline{k}$ of $A \otimes_k \overline{k}$ equipped with the base-changed basis as $\overline{k}$-vector spaces inherited from $x$ and $y$. Since the stabilizer of every point in $\vMod(A\otimes_k \overline{k}, n)$ is smooth by \Cref{prop: smooth stabilizers for representation scheme}, an application of \Cref{lemma: G-orbits vs G(k) orbits} shows that $\widetilde{x}$ and $\widetilde{y}$ lie in the same orbit if and only if there is a matrix $g \in GL_n(\overline{k})$ such that $g \cdot \widetilde{x} = \widetilde{y}$. This is equivalent to $M_x \otimes_{k_x} \overline{k}$ and $M_y \otimes_{k_y} \overline{k}$ being isomorphic as representations of $A \otimes_k \overline{k}$, as desired.
\end{proof}

\begin{lem} \label{lemma: sep extension} 
Let $F \supset k$ be an algebraic field extension.
\begin{itemize}
    \item[(i)] If $M \otimes_k F \cong N \otimes_k F$ as $A\otimes_k F$-modules for finite dimensional $A$-modules $M$ and $N$, then $M \cong N$.
    \item[(ii)] Suppose further that $F \supset k$ is separable and of finite degree. Then every $A\otimes_k F$-module $M$ is isomorphic to a direct summand of $r(M) \otimes_k F$, where $r(M)$ equals $M$ but viewed as an $A$-module via $A\to A\otimes_k F$.
\end{itemize}

\begin{proof} (i) If we choose $F^{\oplus n}$ for the underlying vector space of $M \otimes_k F$ as well as $N \otimes_k F$ for $n\in \mathbb{N}$, then the isomorphism $M \otimes_k F\cong N \otimes_k F$ corresponds to an $n\times n$ matrix $X$ over $F$. Let $L/k$ be the field extension generated by the entries of $X$. Because $F/k$ is an algebraic extension, the degree $[L:k]$ is finite. Now $X$ induces an isomorphism $M \otimes_k L \cong N\otimes_k L$ of $A\otimes_k L$-modules. Restriction of scalars along $A \to  A\otimes_k L$ yields an isomorphism $M^{\oplus [L:k]} \cong N^{\oplus [L:k]}$ of $A$-modules. Counting the multiplicities of indecomposable direct summands in a decomposition of $M^{\oplus [L:k]}$ as well as $N^{\oplus [L:k]}$, it follows that $M\cong N$.

(ii) Since $F\supset k $ is separable and of finite degree, the map $F \otimes_k F \to F, x\otimes y \mapsto xy$ is a split epimorphism of $F$-$F$-bimodules. Applying $M\otimes_F (-)$, it follows that $r(M) \otimes_k F \to M$ is a split epimorphism of $A\otimes_k F$-modules.
\end{proof}
\end{lem}

\subsection{Main results on schemes of representations}

\begin{thm} \label{thm: main result appendix}

    Let $A$ be a finitely generated $k$-algebra, and let $H \subset |\vMod(A,n)|$ be a constructible subset that is preserved by the action of $GL_n$. Suppose that there is an infinite set of non-isomorphic representations of $A$ of dimension $n$ over $k$ which belong to $H$. Then, there is a locally closed subcurve $C \hookrightarrow \vMod(A,n)$ whose underlying topological space is contained in $H$, a positive integer $m >0$ and a set of closed points $\Sigma' \subset |C|_{\leq m}$ satisfying the following:
    \begin{enumerate}
        \item[(i)] $\Sigma'$ has cardinality $|k|$, and so in particular it is infinite.
        \item[(ii)] The points in $\Sigma'$ are in distinct $GL_n$-orbits.
        \item[(iii)] There are $|k|$-many non-isomorphic indecomposable direct summands of the {$A$-modules} $r(M_x)$, where $M_x$ denotes the $n$-dimensional representation of $A \otimes_k k_x$ corresponding to a point $x\in \Sigma'$ with residue field $k_x$ and $r(M_x)$ equals $M_x$ but viewed as an $A$-module via $A \to A\otimes_k k_x$.  
    \end{enumerate}
\end{thm}
\begin{proof}
    By \Cref{lemma: characterization of GL_n orbits} we have that $x,y$ belong to the same $GL_n$-orbit if and only if the corresponding representations of $A \otimes_k \overline{k}$ obtained by tensoring with $\overline{k}$ are isomorphic. In view of \Cref{lemma: sep extension} (i), we conclude that $x,y$ are isomorphic as representations of $A$ if and only if they lie in the same $GL_n$-orbit. Hence, by our assumption we have that $H$ contains infinitely many $k$-points that lie in distinct $GL_n$-orbits. By \Cref{lemma: equivariant constructible subsets}, there is a finite set $X_1, X_2, \ldots, X_{\ell}$ of disjoint $GL_n$-stable locally closed subschemes of $\vMod(A,n)$ such that $H = \bigsqcup_{i=1}^{\ell} |X_i|$. In particular, at least one of these subschemes $X_i$ must contain infinitely many $k$-points that lie in distinct $GL_n$-orbits. Now, by \Cref{prop: smooth stabilizers for representation scheme}, the hypotheses of \Cref{prop: main prop curves} apply for the scheme $X_i$ equipped with the action of $GL_n$, and hence we may apply \Cref{prop: main prop curves} to obtain a curve $C \hookrightarrow X_i \subset \vMod(A,n)$ that fulfills (i) and (ii). 

    Let $\{M_i\}_{i\in I}$ be a set of representatives of all isomorphism classes of indecomposable direct summands of $r(M_x)$ with $x\in \Sigma'$. Suppose that (iii) does not hold, so $|I| < |k|$. By \Cref{lemma: sep extension} (ii), the $A\otimes_k k_x$-module $M_x$ is isomorphic to a direct summand of $r(M_x) \otimes_k k_x$. Thus, the $A \otimes_k \bar{k}$-module $M_x \otimes_{k_x} \bar{k}$ is isomorphic to a direct summand of $r(M_x) \otimes_k \bar{k}$. It follows that every indecomposable direct summand of $M_x \otimes_{k_x} \bar{k}$ is isomorphic to an indecomposable direct summand of $M_i \otimes_k \bar{k}$ with $i \in I$, which there are $|I|$-many of. Since $|\Sigma'| = |k|> |I|$ and the modules $M_x \otimes_{k_x} \bar{k}$ are of the same dimension over $\bar{k}$, it follows from the pigeonhole principle that $M_x \otimes_{k_x} \bar{k} \cong M_y \otimes_{k_y} \bar{k}$ for distinct $x,y \in \Sigma'$. This contradicts that $x,y$ are in distinct $GL_n$-orbits.    
\end{proof}

\end{document}